\documentclass[final,3p,times]{elsarticle}

\usepackage{amsthm}
\usepackage{amsmath}
\usepackage{amsfonts}
\usepackage{amssymb}
\usepackage{amsbsy}
\usepackage{makeidx}
\usepackage{graphicx}
\usepackage{graphics}
\usepackage{shortvrb}
\usepackage{epsf}
\usepackage{psfrag} 
\usepackage{listings}
\usepackage{euscript}
\usepackage{subfig}
\usepackage{commands}
\usepackage{nomencl}
\usepackage{pstricks}
\usepackage[retainorgcmds]{IEEEtrantools}

\begin{document}

\begin{frontmatter}

\author[vki]{Alessandro Munaf\`{o}\fnref{me}\corref{cor1}}
\address[vki]{von Karman Institute for Fluid Dynamics, 1640 Rhode-Saint-Gen\`{e}se, Belgium}
\fntext[me]{PhD candidate, Aeronautics and Aerospace Department, von Karman Institute for Fluid Dynamics, Chauss\'{e}e de Waterloo 72, 1640 Rhode-Saint-Gen\`{e}se, Belgium, \texttt{munafo@vki.ac.be}}
\cortext[cor1]{Corresponding author}
\ead{munafo@vki.ac.be}
\author[texas]{Jeffrey R. Haack\fnref{jeff}}
\address[texas]{The University of Texas at Austin, Austin, TX 78712, USA}
\fntext[jeff]{Postdoctoral Fellow, Department of Mathematics, The University of Texas at Austin, 201 E. 24th Street, Austin, TX 78712, USA, \texttt{haack@math.utexas.edu}}
\author[texas]{Irene M. Gamba\fnref{irene}}
\fntext[irene]{Professor, Department of Mathematics \& The Institute for Computational Engineering and Sciences (ICES), The University of Texas at Austin, 201 E. 24th Street, Austin, TX 78712, USA, \texttt{gamba@math.utexas.edu}}
\author[vki]{Thierry E. Magin\fnref{thierry}}
\fntext[thierry]{Associate Professor, Aeronautics and Aerospace Department, von Karman Institute for Fluid Dynamics, Chauss\'{e}e de Waterloo 72, 1640 Rhode-Saint-Gen\`{e}se, Belgium, \texttt{magin@vki.ac.be}}

\title{A Spectral-Lagrangian Boltzmann Solver for a Multi-Energy Level Gas}

\begin{abstract}
In this paper a spectral-Lagrangian method for the Boltzmann equation for a multi-energy level gas is proposed. Internal energy levels are treated as separate species and inelastic collisions (leading to internal energy excitation and relaxation) are accounted for. The formulation developed can also be used for the case of a mixture of monatomic gases without internal energy (where only elastic collisions occur). The advantage of the spectral-Lagrangian method lies in the generality of the algorithm in use for the evaluation of the elastic and inelastic collision operators. The computational procedure is based on the Fourier transform of the partial elastic and inelastic collision operators and exploits the fact that these can be written as weighted convolutions in Fourier space with no restriction on the cross-section model. The conservation of mass, momentum and energy during collisions is enforced through the solution of constrained optimization problems. Numerical solutions are obtained for both space homogeneous and space in-homogeneous problems. Computational results are compared with those obtained by means of the DSMC method in order to assess the accuracy of the proposed spectral-Lagrangian method.       
\end{abstract}

\end{frontmatter}


\section{Introduction}
Rarefied gas-dynamics has a broad domain of applications ranging from the study of the early phase of spacecraft entry into planetary atmospheres to the investigation of evaporaton and condensation phenomena \cite{Cerc_book,aoki_ev_1994,Frezzotti_2005}.       
    
The degree of rarefaction in a flow depends on the local value of the Knudsen number $Kn$ \cite{Cerc_book}. This is defined as the ratio between the mean free path and a characteristic length of the problem. The higher is the value of the Knudsen number, the more important rarefied gas effects are. When the Knudsen number exceeds values of the order of $0.01$, rarefied gas effects start to become important and attempts to compute rarefied flows by means of a hydrodynamic description based on the Navier-Stokes equations can give inaccurate results. This is precisely due to the failure of Newton's and Fourier's law for the stress tensor and the heat flux vector, respectively, in the rarefied regime \cite{Ferziger_book}. When the medium (gas) is dilute, the Boltzmann equation provides an adequate kinetic description \cite{Cerc_book,Ferziger_book,Chapman_book,Giov_book}. The Boltzmann equation is an integro-differential equation that describes the evolution in the phase-space of the velocity distribution function of the gas species. Once the distribution function known, it is possible to compute macroscopic observables such as density and hydrodynamic velocity by means of suitable moments.          

The solution of the Boltzmann equation by means of numerical techniques represents a computational challenge. This is due to the integro-differential nature of the equation. Another source of difficulty is the high-dimensionality of the problem, since numerical solutions must be sought in the phase-space. In the 1960's the DSMC (Direct-Simulation-Monte-Carlo) method \cite{Bird_book} was developed for obtaining stochastic solutions of the Boltzmann equation. The former is a particle-based technique and has been proven to be accurate \cite{Ivanov_RGD25}. However, it shares the drawbacks of stochastic methods, the main one being the presence of noise in the numerical results. This problem affects, in particular, the accuracy of the solution for low speed and unsteady flows. At the time when the DSMC method was being formulated, the computational power available was quite limited. Deterministic solutions of the Boltzmann equation could be obtained only in the case of model Boltzmann equations \cite{Segal_Ferziger} where the Boltzmann collision operator was replaced by simpler phenomenological expressions (such as that formulated by Bhatnagar, Gross and Krook - BGK model \cite{BGK_paper}). The continuous enhancement of computer performance has encouraged the development of deterministic numerical methods for the Boltzmann equation (with no simplifying assumption on the collision operator). These comprise, among all, discrete velocity models \cite{Tcher,Morris_JCP,josyula_phys_fluids_2011}, numerical kernel methods \cite{ohwada_1993,kosuge_2001} and spectral methods \cite{Bobylev_Rjas,Gamba1,Gamba2,Russo,Jeff1_RGD28}. The main advantage of a deterministic method over the DSMC technique is that the numerical solution obtained is not affected by noise. Deterministic methods can also be applied to flow problems in the hydrodynamic and transition regime, where the use of the DSMC method becomes prohibitively expensive \cite{Kobolov}.   

The purpose of the paper is to extend an existing spectral-Lagrangian method for the Boltzmann equation for a pure gas without internal energy \cite{Gamba1,Gamba2} to a multi-energy level gas. The proposed numerical method can be used for any cross-section model, accounts for elastic and inelastic collisions and allows for the conservation of mass, momentum and energy during collisions. The evaluation of the partial elastic and inelastic collision operators is performed in a fully deterministic manner based on their Fourier transform. In the authors' opinion, this is an important aspect, as in discrete velocity models the same operation is often accomplished stochastically by means of Monte-Carlo integration techniques.    

The paper is structured as follows. In Sect. \ref{sec:phys} the physical model is introduced. In Sect. \ref{sec:phys_Fou} some important features of the Fourier transform of the partial elastic and inelastic collision operators are obtained. The numerical method is described in detail in Sect. \ref{sec:num}. Computational results are given in Sect. \ref{sec:res}. Conclusions are outlined in Sect. \ref{sec:concl}.

\section{Physical model}\label{sec:phys}
\subsection{Assumptions}\label{sec:simpl}
The gas is composed of identical particles with internal degrees of freedom. Based on a quasi-classical approach, it is assumed that the particles may have only certain  discrete internal energy levels (treated as separate species). The indices associated to the internal energy levels (species) are stored in the set $\Sset = \left\{1,\ldots,\Ns\right\}$ (with $\Ns$ being the number of species). The quantities $\mi$, $\Einti$ and $\gi$ indicate, respectively, the mass, the internal energy and the degeneracy of the species $i \in \Sset$. 
 
The following assumptions are introduced for the physical model:
\begin{enumerate}
   \item The gas is dilute and composed of point particles. 
   \item There are no external forces.
   \item The inert particle interactions are binary collisions: 
         \be \label{eq:coll_type}
            \inter, \quad \ijhlset.
         \ee
       \bi 
          \item Elastic collision: $i = k$ and $j = l$.
          \item Inelastic collision: $\jhl \in \Cinset$. The set $\Cinset$ stores the ordered triplets $\jhl$ for all the possible inelastic collisions involving the species $i$ as the first reactant in Eq. \eqnref{eq:coll_type} and is defined as:
           \be \label{eq:in_coll_set}
             \Cinset =  \left\{\jhl \in \left[\Sset \, \times \, \Sset \, \times \, \Sset \, \diagdown \, \bigcup_{s\in\,\Sset} (s, \, i, \, s)\right] \right\}, \quad \iset.
           \ee 
       \ei
      The net internal energy trough the collision is defined by the expression $\DEijhl = \Einth + \Eintl - \Einti - \Eintj$. Notice that the mass is identical for all particles ($m$). The species index (in Eqs. \eqnref{eq:coll_type} - \eqnref{eq:in_coll_set} and in what follows) is kept for generality and to consider the particular case of a mixture of monatomic gases without internal energy.
   \item The reactive collisions are not accounted for. 
\end{enumerate}
\subsection{The Boltzmann equation}
Based on the hypothesis introduced in Sect. \ref{sec:simpl}, a Boltzmann equation can be written for the velocity distribution function $\vdfi (\xb, \Vi, t)$ of the species $i$ (in what follows, only the velocity dependence of the velocity distribution function is explicitly stated):
\be \label{eq:be}
 \fr{\pa \vdfi}{\pa t} + \Vi \cdot \fr{\pa \vdfi}{\pa \xb} = \totQel + \totQin, \quad \iset,
\ee
where the operators in Eq. \eqnref{eq:be} are, respectively, the elastic and inelastic collision operators for the species $i$:
\be
\totQel  = \totQeldef, \quad \totQin = \!\!\!\!\! \totQindef, \quad \iset. \label{eq:Qsumin}
\ee 
The partial elastic and inelastic collision operators are defined as: 
\begin{IEEEeqnarray}{rCl}
\Qel & = & \DomII \left[ \vdfi (\Vip) \, \vdfj (\Vjp) - \vdfi (\Vi) \, \vdfj (\Vj) \right] \sigel \, u \, d\omegabp \, d\Vj, \quad \ijset, \label{eq:Qel} \\
\Qin & = & \DomII \left[ \fr{\gi \, \gj}{\gh \, \gl} \vdfh (\Vip) \, \vdfl (\Vjp) - \vdfi (\Vi) \, \vdfj (\Vj) \right] \sigin \, u \, d\omegabp \, d\Vj, \quad \ijhlsetC. \label{eq:Qin}
\end{IEEEeqnarray} 
In Eqs. \eqnref{eq:Qel} - \eqnref{eq:Qin}, the quantities $\Vi$ and $\Vj$ are, respectively, the velocities of the species $i$ and $j$, $u$ is the relative velocity magnitude $u = \vert \Vi - \Vj \vert$, $\smash{\omegabp}$ is the unit vector along the scattering direction, and $\sigel$ and $\sigin$ are, respectively, the elastic and inelastic differential cross-sections. In Eqs. \eqnref{eq:Qel} - \eqnref{eq:Qin} (and in what follows), primed variables refer to post-collisional values. These are related to pre-collisional values through the conservation of mass, momentum and energy. The elastic and inelastic differential cross-sections ($\sigel$ and $\sigin$, respectively) satisfy the following micro-reversibility relations obtained from the application of Fermi's golden rule \cite{dellacherie}:
\begin{IEEEeqnarray}{rCl} 
\sigel \, u \, d\omegabp \, d\Vj \, d\Vi & = & \sigel \, \up \, d\omegab \, d\Vjp \, d\Vip, \quad \ijset \label{eq:micro_el} \\
\gi \, \gj \, \sigin \, u \, d\omegabp \, d\Vj \, d\Vi & = & \gh \, \gl \, \sigininv \, \up \, d\omegab \, d\Vjp \, d\Vip, \quad \ijhlsetC. \label{eq:micro_in}
\end{IEEEeqnarray}

Equation \eqnref{eq:be} may also be used for the case of a mixture of monatomic gases without internal energy ($\Einti = 0, \, \iset$). In this situation, all collisions are elastic and the inelastic collision operator $\totQin$ in Eq. \eqnref{eq:be} is zero.
\subsection{Collisional invariants}\label{sec:coll_inv}
During an elastic encounter, the number of particles in each internal energy level, the total momentum and the total energy are conserved. This leads to the introduction of the elastic collisional invariants \cite{Giov_book}: 
\begin{IEEEeqnarray}{rCl} 
\collinv{\elrm}{r}{i} & = & \mi \, {\delta}_{i \, r}, \quad r \in\,\Sset, \label{eq:coll_inv_el1}\\
\collinv{\elrm}{\Ns+\nu}{i} & = & \mi \, \val, \quad \nu \, \in \, \left\{1,2,3 \right\}, \, \alset, \label{eq:coll_inv_el_mom}\\
\collinv{\elrm}{\Ns+4}{i} & = & \half \mi \, \vs,\quad \iset, \label{eq:coll_inv_el2}
\end{IEEEeqnarray} 
where the correspondence between the indices $\nu$ and $\alpha$ in Eq. \eqnref{eq:coll_inv_el_mom} is such that $\nu = 1,2,3$ for $\alpha = x,y,z$, respectively. In Eqs. \eqnref{eq:coll_inv_el1} - \eqnref{eq:coll_inv_el2} the symbol ${\delta}_{i \, r}$ stands for Kronecker's delta, while $\val$ and $\vs$ are, respectively, the generic Cartesian component and the magnitude-squared of the velocity vector $\Vi$. After introducing the set of indices for the elastic collisional invariants $\Iel = \left\{1,\ldots,\Ns + 4\right\}$, it is possible to write the relation $\smash{\collinv{\elrm}{\nu}{i}+\collinv{\elrm}{\nu}{j}=\collinv{\elrm}{\nu}{i}+\collinv{\elrm}{\nu}{j} \, (\nu \in \, \Iel})$  in order to express the conservation of the number of particles in each internal energy level, total momentum and energy for the elastic collision $i+j=i+j$. It can be shown that the kernel of the elastic collision operator $\totQel$ in Eq. \eqnref{eq:be} is spanned by the set of elastic collisional invariants \cite{Giov_book}: 
\be \label{eq:Qel_inv} 
\sumins \, \DomI \!\!\! \collinv{\elrm}{\nu}{i} \, \totQel \, d\Vi = 0, \quad \nu \in\,\Iel.
\ee

During an inelastic encounter, due to the transitions among the internal energy levels, the total number of particles, momentum and energy are conserved. This leads to the introduction of the inelastic collisional invariants \cite{Giov_book}:    
\begin{IEEEeqnarray}{rCl}
\collinv{\inrm}{1}{i} & = & \mi, \label{eq:coll_inv_in1}\\
\collinv{\inrm}{1 + \nu}{i} & = & \mi \, \val, \quad \nu \, \in \, \left\{1,2,3 \right\}, \, \alset, \\
\collinv{\inrm}{5}{i} & = & \half  \mi \, \vs + \Einti, \quad \iset. \label{eq:coll_inv_in2}
\end{IEEEeqnarray}
After introducing the set of indices for the inelastic collisional invariants $\Iin = \left\{1,\ldots,5\right\}$, it is possible to write the relation $\smash{\collinv{\inrm}{\nu}{i}+\collinv{\inrm}{\nu}{j}=\collinv{\inrm}{\nu}{k}+\collinv{\inrm}{\nu}{l} \, (\nu \in \, \Iin)}$  in order to express the conservation of the total number of particles, momentum and energy for the inelastic collision $\inter$. It can be shown that the kernel of the inelastic collision operator $\totQin$ in Eq. \eqnref{eq:be} is spanned by the set of inelastic collisional invariants \cite{Giov_book}: 
\be \label{eq:Qin_inv} 
\sumins \, \DomI \!\!\!\! \collinv{\inrm}{\nu}{i}  \, \totQin \, d\Vi = 0, \quad \nu \in\,\Iin.
\ee
\subsection{Conserved macroscopic moments}\label{sec:mom_cons}
The elastic and inelastic collisional invariants (Eqs. \eqnref{eq:coll_inv_el1} - \eqnref{eq:coll_inv_el2}  and Eqs. \eqnref{eq:coll_inv_in1} - \eqnref{eq:coll_inv_in2}) introduced in Sect. \ref{sec:coll_inv} allows for the introduction of the following flow macroscopic quantities (in the hydrodynamic frame) as average microscopic quantities: 
\begin{IEEEeqnarray}{rClrCl}
\rhoj & = & \sumins \, \DomI \!\!\!\! \collinv{\elrm}{i}{j} \vdfj(\Vi) \, d\Vi, \quad \jset, \quad &  \rho & = & \sumins \, \DomI \!\!\!\! \collinv{\inrm}{1}{i} \vdfi(\Vi) \, d\Vi, \label{eq:dens_macro} \\
\rho \, V_{\alpha} & = & \sumins \, \DomI \!\!\!\! \collinv{\elrm}{\Ns + \nu}{i} \vdfi(\Vi)\, d\Vi, \quad & \rho \, V_{\alpha} & = & \sumins \, \DomI \!\!\!\! \collinv{\inrm}{1 + \nu}{i} \vdfi(\Vi) \, d\Vi, \label{eq:mom_macro} \\
\rho \, \etr  + \rho \fr{V^{\, 2}}{2}  & = & \sumins \, \DomI \!\!\!\! \collinv{\elrm}{\Ns + 4}{i} \vdfi(\Vi)\, d\Vi, \quad \quad & \rho \, \etr +  \rho \, \eint + \rho \fr{V^{\, 2}}{2}  & = & \sumins \, \DomI \!\!\!\! \collinv{\inrm}{4}{i} \vdfi(\Vi)\, d\Vi, \label{eq:en_macro}
\end{IEEEeqnarray}
with $\alset$ and $\nu \, \in \, \left\{1,2,3\right\}$ in Eq. \eqnref{eq:mom_macro}. In Eqs. \eqnref{eq:dens_macro} - \eqnref{eq:en_macro}, the quantity $\rhoj$ is the density of the species $j$, $\smash{\rho = \sumins \rhoi}$ is the gas (or mixture) density, $V_{\alpha}$ and $V^{\, 2}$ are, respectively, the generic Cartesian component and the magnitude-squared of the hydrodynamic velocity vector $\Vvec$, while $\etr$ and $\eint$ are, respectively, the gas specific translational and internal energy. The macroscopic moments defined in Eqs. \eqnref{eq:dens_macro} - \eqnref{eq:en_macro} represent the quantities that are conserved in a flow in view of the properties satisfied by the elastic and inelastic collision operators given in Eq. \eqnref{eq:Qel_inv} and Eq. \eqnref{eq:Qin_inv}, respectively.   
\subsection{The Maxwell-Boltzmann velocity distribution function}\label{sec:equil}
Under thermodynamic equilibrium conditions, the solution of the Boltzmann equation (Eq. \eqnref{eq:be}) is given by the Maxwell-Boltzmann velocity distribution function \cite{Ferziger_book}:
\be \label{eq:MB} 
f^{\, \eq}_i (\Vi) = \fr{\rhoieq}{\mi} \left(\fr{\mi}{2 \, \pi \, \kb  \, \Teq}\right)^{3/2} \!\!\!\!\! \exp \left(-\fr{\mi \left|\Vi - \Vvec \right|^{\, 2}}{2 \, \kb \, \Teq} \right), \quad \iset,
\ee
where the superscript $\eq$ stands for equilibrium. In Eq. \eqnref{eq:MB}, $\Teq$ is the gas (equilibrium) temperature and $\kb$ Boltzmann's constant. The density of the species $i$ at equilibrium is obtained through the Boltzmann distribution law: 
\be \label{eq:Boltz_distr}
\fr{\rhoieq}{\rhoeq} = \fr{\gi \exp\left( -\fr{\Einti}{\kb \, \Teq}\right)}{\Qint}, \quad \iset,
\ee
where the gas internal partition function is given by:
\be 
\Qint = \sumins \! \gi \exp\left(-\fr{\Einti}{\kb \, \Teq}\right), 
\ee
When the velocity distribution function is Maxwell-Boltzmann (Eq. \eqnref{eq:MB}), the gas specific translational and internal energy are given in the expressions \cite{Ferziger_book}:  
\begin{IEEEeqnarray}{rCl}
\etreq  & = & \fr{1}{\rhoeq} \sumins \fr{3}{2} \nbieq \, \kb \, \Teq, \label{eq:gas_tr} \\
\einteq & = & \fr{1}{\Qint} \sumins \! \fr{\Einti}{\mi} \gi \exp\left(-\fr{\Einti}{\kb \, \Teq}\right), \label{eq:gas_int}
\end{IEEEeqnarray}   
where the number density of the species $i$ in Eq. \eqnref{eq:gas_tr} is $\nbieq = \rhoieq/\mi$. The gas pressure is given by Dalton's law of partial pressures, $\peq = \sumins \nbieq \, \kb \, \Teq$.       

The above equations (with the exception of Eq. \eqnref{eq:Boltz_distr}) can be also used for the particular case of a mixture of monatomic gases without internal energy. In this situation, the gas specific internal energy (Eq. \eqnref{eq:gas_int}) is zero. 
\subsection{Non-equilibrium}\label{sec:neq} 
Outside of equilibrium conditions, translational and internal temperatures are introduced: 
\bi
   \item Species translational temperature components: 
    \be \label{eq:neq_mom_Tr1}
      \Tial = \fr{\mi}{\nbi \, \kb} \!\! \DomI \!\!\!\! \left(\val - \Val \right)^{\, 2} \vdfi(\Vi) \, d\Vi, \quad \iset, \, \alset.
    \ee
   \item Species translational temperature:
    \be 
      \Ti = \fr{1}{3} \sum_{\alset} \!\!\!\! \Tial, \quad \iset.
    \ee
   \item Translational temperature components:
    \be \label{eq:neq_mom_Tr}
      \Tal = \fr{1}{n}  \sumins \nbi \, \Tial, \quad \alset.
    \ee
   \item Translational temperature:
    \be \label{eq:neq_mom_Tr2}
      T = \fr{1}{3} \sum_{\alset} \!\!\!\! \Tal, \quad \iset.
    \ee
   \item Internal temperature:
    \be \label{eq:neq_mom_Tint}
      \sumins \! \nbi \, \Einti  = \rho \, \einteq(\Tint).
    \ee
\ei 
The gas number density in Eq. \eqnref{eq:neq_mom_Tr} is $n = \sumins \nbi$. Notice that Eq. \eqnref{eq:neq_mom_Tint} only provides an implicit definition for the internal temperature. This is due to the fact that the gas specific internal energy (Eq. \eqnref{eq:gas_int}) is a non-linear function of the temperature. The gas pressure is always computed by means of Dalton's law of partial pressures, $p = \sumins \nbi \, \kb \, \Ti$. Other macroscopic moments of interest in non-equilibrium conditions are:
\bi 
   \item Species diffusion velocity:
      \be 
        \Vdiffi = \fr{1}{\nbi} \!\!\!\! \DomI \!\!\!\! (\Vi - \Vvec) \vdfi(\Vi) \, d\Vi, \quad \iset.
      \ee
   \item Viscous stress tensor:
      \be 
        \taumix = \!\! \sumins \, \DomI \!\!\!\! \mi \left(\Vi - \Vvec \right) \otimes \left(\Vi - \Vvec \right) \vdfi (\Vi) \, d\Vi - p \, \mbf{I},
      \ee
      where $\mbf{I}$ is the second order identity tensor.
   \item Heat flux vector:  
      \be 
         \qmix = \!\! \sumins \, \DomI \!\!\!\! \left(\Vi - \Vvec \right) \left( \half \mi \left| \Vi - \Vvec \right|^{\, 2} + \Einti\right) \vdfi(\Vi) \, d\Vi.
      \ee
\ei 
\section{The Fourier transform of the partial elastic and inelastic collision operators}\label{sec:phys_Fou}
The numerical method proposed in Sect. \ref{sec:num} makes use of the Fourier transform  of the partial elastic and inelastic collision operators (Eqs. \eqnref{eq:Qel} - \eqnref{eq:Qin}). The starting point is the weak form of the partial elastic and inelastic collision operators \cite{Ferziger_book,Bobylev_Fourier}: 
\begin{IEEEeqnarray}{rCl} 
\DomI \!\!\!\!\! \Phii (\Vi) \, \Qel \, d\Vi & = & \!\! \DomIII \!\! \left[\Phii (\Vip) - \Phii (\Vi)\right] \vdfi(\Vi) \, \vdfj(\Vj) \, \sigel \, u \, d\omegabp \, d\Vj \, d\Vi,\quad i,j\in\,\Sset, \label{eq:wQel} \\
\DomI \!\!\!\!\! \Phii (\Vi) \, \Qin \, d \Vi & = & \!\! \DomIIIp \!\! \Phii (\Vi) \, \vdfh (\Vip) \, \vdfl (\Vjp) \, \sigininv \, \up \, d\omegab \, d\Vjp \, d\Vip \,\, - \,  \DomIII \!\! \Phii (\Vi) \, \vdfi (\Vi) \, \vdfj (\Vj) \, \sigin \, u \, d\omegabp \, d\Vj \, d\Vi, \nonumber \\
&&\ijhlsetC, \label{eq:wQin}
\end{IEEEeqnarray}
where the function $\Phii (\Vi)$ in Eqs. \eqnref{eq:wQel} - \eqnref{eq:wQin} is a smooth test function of the velocity vector $\Vi$. The substitution of a Fourier velocity mode $\Phii (\Vi) = (2\pi)^{-3/2} \exp (- \imath \, \boldsymbol{\zeta} \cdot \Vi)$ in Eqs. \eqnref{eq:wQel} - \eqnref{eq:wQin} gives the Fourier transform of the partial elastic and inelastic collision operators. 
\newtheorem*{WF}{Remark}
\begin{WF}
For the partial elastic collision operator $\Qel$ the weak form (Eq. \eqnref{eq:wQel}) is obtained by applying the usual technique of swapping between primed and un-primed variables in the integral and by exploiting micro-reversibility (Eq. \eqnref{eq:micro_el}). Since in an elastic collision there are no transitions between the internal energy levels, swapping between primed and un-primed variables has no effect on the species index. This allows for casting the weak form into a unique integral unloving the species velocity distribution function in the pre-collision state. The same result cannot be obtained the case of an inelastic collision (swapping between primed and un-primed variables leads to a species index change). The weak form of the partial inelastic collision operator $\smash{\Qin}$ (Eq. \eqnref{eq:wQin}) is obtained by applying micro-reversibility (Eq. \eqnref{eq:micro_in}) to the gain part of the operator, while the loss part is left unchanged. As discussed by Dellacherie \cite{dellacherie}, alternative expressions to that given in Eq. \eqnref{eq:wQin} can be obtained. However, the one given in Eq. \eqnref{eq:wQin} is the most suited for the present work.
\end{WF}
\newtheorem{FT}{Proposition}[section]
\begin{FT}
The Fourier transform $\FQel (\fa)$ and $\FQin (\fa)$ of the partial elastic and inelastic collision operators ($\Qel$ and $\Qin$, respectively) can be written as weighted convolutions in Fourier space:
\begin{IEEEeqnarray}{rCl} 
\FQel (\fa) & = & \!\!\! \DomIF \!\!\!\! \Fvdfi \left(\fa - \fb \right) \, \Fvdfj (\fb) \, \FFGel \left(\fa,\fb \right) d\fb, \quad \ijset, \label{eq:FT_theor}\\ 
\FQin (\fa) & = & \!\!\! \DomIF \!\!\!\! \left[ \Fvdfh \left(\fa - \fb \right) \, \Fvdfl (\fb) \, \FFGinG \left(\fa,\fb \right) - \Fvdfi \left(\fa - \fb \right) \, \Fvdfj (\fb) \, \FFGinL \left(\fb \right) \right] d\fb, \quad \ijhlsetC. \label{eq:FT_theor2}
\end{IEEEeqnarray}
In Eqs. \eqnref{eq:FT_theor} - \eqnref{eq:FT_theor2}, the quantities $\Fvdfi$, $\Fvdfj$, $\Fvdfh$ and $\Fvdfl$ are, respectively, the Fourier transform of the velocity distribution functions of the species $i$, $j$, $k$ and $l$, respectively, while the functions $\smash{\FFGel\left(\fa,\fb \right)}$, $\smash{\FFGinG\left(\fa,\fb \right)}$ and $\smash{\FFGinL\left(\fb \right)}$ are convolution weights defined as:
\begin{IEEEeqnarray}{rCl}
\FFGel \left(\fa,\fb \right) & = & \scale \DomIIuom \!\! u \, \sigel \, \left\{ \exp\left[- \imath \,  \fr{\muij}{\mi} \fa \cdot \left(\urelp - \urel  \right) \right] - 1\right\} \, \exp\left(- \imath \, \fb \cdot \urel \right)  \,  d\omegabp \, d\urel, \quad \ijset, \label{eq:Wel} \\
\FFGinG \left(\fa,\fb \right) & = & \scale \DomIIuomp \!\! \up \, \sigininv \, \exp\left[- \imath \,  \fr{\muij}{\mi} \fa \cdot \left(\urel - \urelp \right) \right] \, \exp\left(- \imath \, \fb \cdot \urelp \right)  \,  d\omegab \, d\urelp,  \label{eq:Win1} \\
\FFGinL \left(\fb \right) & = &  \scale \DomIIuom \!\! u \, \sigin \, \exp\left(- \imath \, \fb \cdot \urel \right)  \,  d\omegabp \, d\urel, \quad \ijhlsetC. \label{eq:Win} 
\end{IEEEeqnarray}
where the symbol $\muij$ in Eqs. \eqnref{eq:Wel} - \eqnref{eq:Win1} stands for the reduced mass of the species $i$ and $j$, $\muij = \mi \, \mj /(\mi + \mj)$. Notice that in obtaining Eq. \eqnref{eq:Win1} the relation $\muij = \muipjp $ (valid for a multi-energy level gas) has been exploited. 
\end{FT}
\begin{proof}
The direct substitution of $\Phii (\Vi) = (2\pi)^{-3/2} \exp (- \imath \, \boldsymbol{\zeta} \cdot \Vi)$ in Eqs. \eqnref{eq:wQel} - \eqnref{eq:wQin} gives, after some algebra, the thesis.
\end{proof}

\noindent The above proposition, leads to the following observations:
\begin{enumerate}
   \item The convolution weights $\smash{\FFGel \left(\fa,\fb \right)}$, $\smash{\FFGinG \left(\fa,\fb \right)}$ and $\smash{\FFGinL \left(\fb \right)}$ in Eqs. \eqnref{eq:Wel} - \eqnref{eq:Win} only depend on the differential cross-section. No dependence on the value of the species velocity distribution function occurs. This fact can be exploited to develop a computational method (see Sect. \ref{sec:num}) that makes use of Eqs. \eqnref{eq:FT_theor} - \eqnref{eq:FT_theor2} for the numerical evaluation of the collision operators (the weights associated to each collision can be pre-computed). 
   \item The convolution weights $\smash{\FFGinG \left(\fa,\fb \right)}$ and $\smash{\FFGinL \left(\fb \right)}$ in Eq. \eqnref{eq:FT_theor2} are associated to the gain and loss part of the partial inelastic collision operator $\Qin$ (Eq. \eqnref{eq:Qin}) and cannot be directly summed to give a unique convolution weight. This operation is possible only when the collision is elastic. For this case, it can be shown that $\smash{\FFGel \left(\fa,\fb \right) = {\hat{G}}^{\, ij}_{ij}(\fa,\fb) - {\hat{L}}^{\, ij}_{ij}(\fb)}$.  
  \item Since in the definition provided by Eqs. \eqnref{eq:Wel} - \eqnref{eq:Win} no assumption is made on the differential cross-section, anisotropic interactions can also be taken into account.
\end{enumerate} 

\noindent In the case of isotropic interactions (differential cross-section depending only on the relative velocity magnitude $u$), the mathematical expressions for the convolution weights $\smash{\FFGel \left(\fa,\fb \right)}$, $\smash{\FFGinG \left(\fa,\fb \right)}$ and $\smash{\FFGinL \left(\fb \right)}$ given in Eqs. \eqnref{eq:Wel} - \eqnref{eq:Win} simplify.   
\newtheorem{Welin}[FT]{Proposition}
\begin{Welin}
The convolution weights $\smash{\FFGel \left(\fa,\fb \right)}$, $\smash{\FFGinG \left(\fa,\fb \right)}$ and $\smash{\FFGinL \left(\fb \right)}$ appearing in the Fourier transform $\FQel (\fa)$ and $\FQin (\fa)$ of the partial elastic and inelastic collision operators ($\Qel$ and $\Qin$, respectively), reduce to one-dimensional integrals on the pre and post-collisional relative velocity magnitudes ($u$ and $u^{\prime}$, respectively) in the case of isotropic interactions:
\begin{IEEEeqnarray}{rCl}
\FFGel \left(\fa,\fb \right) & = &  4 \sqrt{2 \, \pi} \!\!\!\!\!\! \int\limits_{u \, \in \, [0, +\infty)}  \!\!\!\!\!\!\! \sigel \left[\bess \left( \zeta \, \fr{\muij}{\mi} u \right) \, \bess \left(\left| \fb  - \fa \, \fr{\muij}{\mi} \right| u \right) - \bess \left(\xi \, u \, \right) \right] \uc \, du, \quad \ijset, \label{eq:WEliso} \\
\FFGinG \left(\fa,\fb \right) & = & 4 \sqrt{2 \, \pi} \!\!\!\!\!\!\!\!\!\! \int\limits_{\up \, \in \, [\ustarG, +\infty)}  \!\!\!\!\!\!\!\!\!\!\! \sigininv \bess \left( \zeta \, \fr{\muij}{\mi} \sqrt{ \ups + 2\, \DEijhl / \muij} \right) \, \bess \left(\left| \fb  - \fa \, \fr{\muij}{\mi} \right| \up \right) \upc \, d\up,  \label{eq:WIniso1}\\ 
\FFGinL \left(\fb \right) & = &  4 \sqrt{2 \, \pi} \!\!\!\!\!\!\!\! \int\limits_{u \, \in \, [\ustarL, +\infty)}  \!\!\!\!\!\!\!\!\! \sigin \bess \left(\xi \, u \, \right) \uc \, du, \quad \ijhlsetC.  \label{eq:WIniso2}
\end{IEEEeqnarray}
In Eqs. \eqnref{eq:WEliso} - \eqnref{eq:WIniso2}, the function $\bess(x) = \sin(x)/x$ is the zero-order spherical Bessel function (or un-normalized sinc function) while the quantities $\zeta$ and $\xi$ are, respectively, the magnitudes of the vectors $\fa$ and $\fb$. The lower limits $\ustarG$ and $\ustarL$ for the integrals defining the gain and loss inelastic convolution weights ($\smash{\FFGinG \left(\fa,\fb \right)}$ and $\smash{\FFGinL \left(\fb \right)}$, respectively) in Eqs. \eqnref{eq:WIniso1} - \eqnref{eq:WIniso2} are:
\be
\ustarG = 
  \begin{cases}
   \sqrt{-\fr{2 \, \DEijhl}{\muij}} & {if \quad} \DEijhl < 0, \\
   0                             & {if \quad} \DEijhl \geq 0,
  \end{cases} 
\quad
\ustarL =
  \begin{cases}
   \sqrt{\fr{2 \, \DEijhl}{\muij}} & {if \quad} \DEijhl > 0, \\
   0                             & {if \quad} \DEijhl \leq 0.
  \end{cases} 
\ee 
\end{Welin}
\begin{proof}
The use of a spherical coordinate system in the integrals over $\uvec$, $\uvecp$, $\omegab$ and $\omegabp$ in Eqs. \eqnref{eq:Wel} - \eqnref{eq:Win} gives, after some algebra, the thesis. As an example, the integral over $\omegabp$ in Eq. \eqnref{eq:Win} defining the convolution weight $\smash{\FFGinL \left(\fb \right)}$ can be computed by adopting a spherical coordinate system for the vector $\omegabp$ with the pole aligned along the direction of the vector $\fb$. A similar procedure can be used for the other integrals in Eqs. \eqnref{eq:Wel} - \eqnref{eq:Win}.  
\end{proof} 
\section{Numerical method}\label{sec:num}
The numerical method proposed for solving the Boltzmann equation (Eq. \eqnref{eq:be}) exploits the particularly simple structure assumed by the Fourier transform of the partial elastic and inelastic collision operators (weighted convolution in Fourier space - Eqs. \eqnref{eq:FT_theor} - \eqnref{eq:FT_theor2}). Only zero/one-dimensional flows are considered and the velocity space is always kept three-dimensional. This can be justified in view of the fact that the main purpose of the paper is to develop an algorithm for the evaluation of the collision operators allowing for the conservation of mass, momentum and energy during collisions (Eqs. \eqnref{eq:Qel_inv} and \eqnref{eq:Qin_inv}). The extension of the method to multi-dimensional flows is trivial, as the aforementioned algorithm remains the same whether the flow is multi-dimensional or not.

In the case when the flow is one-dimensional and its direction is aligned with the $x$ axis of a Cartesian reference frame $(O;x,y,z)$, the Boltzmann equation (Eq. \eqnref{eq:be}) becomes:
\be \label{eq:be1d}
\fr{\pa \vdfi}{\pa t} + \vx \fr{\pa \vdfi}{\pa x} = \totQel + \totQin, \quad \iset.
\ee           
In order obtain numerical solutions to Eq. \eqnref{eq:be1d}, the following steps have to be taken:
\begin{enumerate}
  \item Discretization of the phase-space,
  \item Choice of a time-marching method, 
  \item Development of a computational algorithm for an efficient evaluation of the collision operators in Eq. \eqnref{eq:be1d} allowing to satisfy the conservation requirements stated in Eqs. \eqnref{eq:Qel_inv} and \eqnref{eq:Qin_inv}.
\end{enumerate}
All the items of the previous list are described in Sects. \ref{sec:phase} - \ref{sec:Coll_alg}.
\subsection{Discretization of the phase-space}\label{sec:phase}
A Cartesian reference frame $(O;\vx,\vy,\vz)$ is introduced for the velocity space. The former is discretized by considering points falling inside a cube (with side semi-length $\Lv$) centered at the origin $O$:
\be \label{eq:set_vel}
\Vcal = \left\{\vvec = (\vx,\vy,\vz) \, \in \, \Rd \, \vert \, v_{\, \alpha} \, \in \, [-\Lv,\Lv), \, \alset \right\}.
\ee
The individual discrete velocity nodes belonging to the set $\Vcal$ in Eq. \eqnref{eq:set_vel} are obtained as follows. Let $\Dv$ be the velocity mesh spacing, defined as:
\be \label{eq:Dv}
\Dv = \fr{2 \, \Lv}{\Nv},
\ee
where $\Nv$ is the number of velocity nodes along the $\vx$, $\vy$ and $\vz$ directions, let $\kvec = (\kx, \ky, \kz)$ be the vector of indices associated to the discrete velocity node $\vk = (\vkx, \vky, \vkz)$ and let $\Vset$ be the set $\Vset = \left\{0,\ldots,\Nv - 1\right\}$. The discrete velocity node $\vk$ belonging to the set $\Vcal$ is computed as follows:
\be \label{eq:vk}
\vk = -\Lv \left(\ivx + \ivy  + \ivz \right) + \kvec \, \Dv, \quad \kvec = (\kx,\kz,\ky)\in\,\Vsetc. 
\ee 
In Eq. \eqnref{eq:vk}, the vectors $\ivx$, $\ivy$ and $\ivz$ are, respectively, the unit vectors of the $\vx$, $\vy$ and $\vz$ axes of the Cartesian frame $(O;\vx,\vy,\vz)$, and the set $\Vsetc$ is defined as $\Vsetc = \Vset \, \xrm \, \Vset \, \xrm \, \Vset$. A vector of integration weights $\smash{\wkset = ( \wkx, \wky, \wkz)}$ is introduced and associated to each discrete velocity node $\vk$.  

As mentioned before, the algorithm proposed for the evaluation of the collision operators (given in Sect. \ref{sec:Coll_alg}) is based on the Fourier transform of the former (Eqs. \eqnref{eq:FT_theor} - \eqnref{eq:FT_theor2}). This is the reason why a Fourier velocity space (associated to the velocity space described above) is introduced and discretized as follows. A Cartesian reference frame $(O;{\zeta}_x,{\zeta}_y,{\zeta}_z)$ in the Fourier velocity space is introduced and  the points falling inside a cube (with semi-length $\Leta$) centered at the origin $O$ are considered:
\be \label{eq:set_velF}
\VcalF = \left\{\fa = (\zetax,\zetay,\zetaz) \, \in \Rd \, \vert \, {\zeta}_{\, \alpha} \, \in [-\Leta,\Leta), \, \alset \right\}
\ee  
The discrete Fourier velocity nodes belonging to the set $\VcalF$ in Eq. \eqnref{eq:set_velF} are obtained as follows. Let $\Deta$ be the Fourier velocity mesh spacing, defined as:
\be \label{eq:DvF} 
\Deta = \fr{2 \, \Leta}{\Nv},
\ee
and let $\boldsymbol{\eps} = (\epsx, \epsy, \epsz)$ be the vector of indices associated to the discrete Fourier velocity node $\ze = (\zex, \zey, \zez)$. The discrete Fourier velocity node $\ze$ belonging to the set $\VcalF$ is computed as follows:
\be \label{eq:epsk}
\ze = -\Leta \left(\izetax + \izetay  + \izetaz \right) + \boldsymbol{\eps} \, \Deta, \quad \boldsymbol{\eps} = (\epsx,\epsy,\epsy)\in\,\Vsetc.  
\ee   
In Eq. \eqnref{eq:epsk}, the vectors $\izetax$, $\izetay$ and $\izetaz$ are, respectively, the unit vectors of the $\zetax$, $\zetay$ and $\zetaz$ axes of the Cartesian frame $(O;{\zeta}_x,{\zeta}_y,{\zeta}_z)$. A vector of integration weights $\smash{\wepsset = ( \wepsx, \wepsy, \wepsz)}$ is introduced and associated to each Fourier velocity node $\ze$.

In the present work, the semi-length $\Lv$ and the number of nodes $\Nv$ along each direction of the velocity space are considered as input parameters. The velocity mesh spacing $\Dv$ is then computed through Eq. \eqnref{eq:Dv}. The semi-length $\Leta$ and the mesh spacing $\Deta$ of the Fourier velocity space are found by imposing in Eq. \eqnref{eq:DvF} the condition:
\be \label{eq:FFT}
\Deta \, \Dv = \fr{2 \, \pi}{\Nv}.
\ee
The substitution of the expressions for $\Dv$ and $\Deta$ (Eq. \eqnref{eq:Dv} and Eq. \eqnref{eq:DvF}, respectively)  in Eq. \eqnref{eq:FFT} leads to:
\be \label{eq:Leta}
\Leta = \fr{\pi \Nv}{2 \, \Lv}.
\ee
In Eq. \eqnref{eq:Leta}, the semi-length $\Leta$ is completely determined from the values of the input parameters ($\Nv$ and $\Lv$). Once $\Leta$ computed, the Fourier velocity mesh spacing $\Deta$ is then found from Eq. \eqnref{eq:DvF}. The choose of a uniform mesh along each direction of the velocity spaces (physical and Fourier) and of the condition given by Eq. \eqnref{eq:FFT} are due to the use of the Fast-Fourier-Transform (FFT) algorithm \cite{Gamba1,Gamba2} for the evaluation of the Fourier and the inverse Fourier transforms. 

The position space is discretized by considering points belonging to the following subset $\Xcal$ of the $x$ axis:
\be \label{eq:domx}
\Xcal = \left\{ (x,0,0) \, \in \, \Re \, \vert \, x \, \in \, [-\Lmx,\Lpx] \right\},
\ee
where the quantities $\Lmx$ and $\Lmx$ in Eq. \eqnref{eq:domx} are both positive. A finite volume grid can be defined based on Eq. \eqnref{eq:domx}. Let $\Nx$ be the number of nodes in the position space, $\srm$ be the index corresponding to the node $\xs$ in the discretized position space and $\Xset$ the set $\Xset = \left\{0,\ldots,\Nx - 2\right\}$. The centroid location $\xcs$ and the volume $\Dcs$ of the cell $\srm$ (volume) contained between the nodes $\srm$ and $\srm + 1$ are computed as:
\begin{IEEEeqnarray}{rCl}
\xcs & = & \fr{1}{2}(\xsp + \xs),\\
\Dcs & = & \xsp - \xs, \quad \srm \in \Xset. \label{eq:vol}
\end{IEEEeqnarray}

The time domain is discretized as follows. Let $\NT$ be the number time-steps, $\Delta t_n$ the time-step value associated to the time-level $\tn$ and $\Tset$ the set $\Tset = \left\{0,\ldots, \NT\right\}$. The set of nodes of the discretized time-domain is then:
\be \label{eq:domt}
\Tcal = \left\{\tn = \sum_{m\leq n} \! \Delta t_m \, \in \, \Re \, \vert n,m \, \in \, \Tset \right\}.
\ee

For sake of later convenience, it is useful to introduce the following compact notation for the value of the velocity distribution function of the species $i$ at the point $(\xcs, \vk)$ of the discretized phase-space at the discrete time-level $\tn$: 
\be 
\fnum = \vdfi (\xcs, \vk, \tn), \quad \iset, \, \vk \, \in \, \Vcal,\, \xcs \, \in \, \Xcal,\, \tn \, \in \, \Tcal.
\ee 

\subsection{Time-marching method}\label{sec:op_split}
In order to obtain numerical solutions to Eq. \eqnref{eq:be1d}, the methods of lines is employed \cite{Hirsch_book}. The Finite volume method is firstly applied to Eq. \eqnref{eq:be1d} (written for each discrete velocity node as given in Eq. \eqnref{eq:vk}) in order to perform the discretization in the position space. Secondly, the semi-discrete set of equations obtained is integrated in time by means of a time-marching method. In the present work, explicit time-integration methods are considered due their ease of implementation and low memory requirements when compared with implicit methods \cite{Mieussens_JCP,Mieussens_Struchtrup}. 

The application of the Finite volume method to Eq. \eqnref{eq:be1d} written for the discrete velocity node $\vk$ leads to the following semi-discrete equation:
\be \label{eq:adv2}
\Dcs \fr{\pa \fsijk}{\pa t} + \Fluxp - \Fluxm = \Dcs \, {\Qi}_{\, \srm \, \kvec}, \quad \iset, \, s \, \in \, \Xset, \, \kvec \, \in \, \Vsetc,
\ee   
where $\smash{\Fluxp}$ and $\smash{\Fluxm}$ are, respectively, the numerical fluxes at the interfaces $\srm + 1/2$ and $\srm - 1/2$ of the cell $\srm$, $\Dcs$ is the volume of the cell $\srm$ (Eq. \eqnref{eq:vol}) and ${\Qi}_{\rm \, s\, \kvec}$ represents the sum of the elastic and inelastic collision operators for the species $i$ evaluated at the node $(\xcs,\vk)$ of the discretized phase-space (the algorithm for its numerical evaluation is explained in Sect. \ref{sec:Coll_alg})  The numerical flux $\Fluxp$ in Eq. \eqnref{eq:adv2} is computed by means of a second order slope-limited upwind scheme \cite{Hirsch_book}:
\be \label{eq:flux}
\Fluxp = \apk \, \fLsijk + \amk \, \fRspijk, \quad \iset, \, s \,\in \, \Xset, \, \kvec \, \in \, \Vsetc,
\ee 
where $\apk$ and $\amk$ in Eq. \eqnref{eq:flux} are, respectively, the positive and negative wave speeds:
\begin{IEEEeqnarray}{rCl}
\apk & = & \max(\vkx, 0),\\
\amk & = & \min(\vkx, 0), \quad \kx \, \in \, \Vset ,
\end{IEEEeqnarray}
and $\fLsijk$ and $\fRspijk$ are the reconstructed values of the distribution function at the left and right sides, respectively, of the interface $\srm + 1/2$ between the cells $\srm$ and $\rm s + 1$. The reconstructed values of the distribution functions ($\fLsijk$ and $\fRspijk$ in Eq. \eqnref{eq:flux}) are obtained by means of a limited MUSCL reconstruction \cite{vanLeer}: 
\begin{IEEEeqnarray}{rCl}
\fLsijk & = &  \fsijk + \half \phi (\rLrm) \left(\fsijk - \fsmijk\right),  \label{eq:rec1} 
\\
\fRspijk & = & \fspijk -\half \phi (\rRrm) \left(\fsppijk - \fspijk\right), \quad \iset, \, s \, \in \, \Xset, \, \kvec \, \in \, \Vsetc.  \label{eq:rec2} 
\end{IEEEeqnarray} 
In Eqs. \eqnref{eq:rec1} - \eqnref{eq:rec1}, $\phi(r)$ is a slope limiter function (such as those proposed by van Albada, van Leer \textit{et al} \cite{Hirsch_book}) and $\rLrm$ and $\rRrm$ are, respectively, the left and right ratios of consecutive differences:
\begin{IEEEeqnarray}{rCl} 
\rLrm & = & \fr{\fspijk - \fsijk}{\fsijk - \fsmijk}, \\
\rRrm & = & \fr{\fspijk - \fsijk}{\fsppijk - \fspijk}, \quad \iset, \, s \, \in \, \Xset, \, \kvec \, \in \, \Vsetc.
\end{IEEEeqnarray}   
Equation \eqnref{eq:adv2} is integrated in time by means of the Forward Euler method \cite{Hirsch_book}:
\be \label{eq:fe}
\fnump = \fnum -\fr{\Dts}{\Delta \xs}\left[\left(\Fluxpn - \Fluxmn \right) - \Dcs \, {Q_{i}}^{\, n}_{\, \srm \, \mbf{k}} \right],\quad \iset, \, s \, \in \, \Xset, \, \kvec \, \in \, \Vsetc, \, n \, \in \, \Tset.
\ee
The time-step $\Dts$ in Eq. \eqnref{eq:fe} is computed according to \cite{Mieussens_JCP}:
\be \label{eq:time_step}
\Dts = \fr{\mrm{CFL}}{\fr{1}{\Dtc} + \fr{\Lv}{\Dcs}}, \quad  s \, \in \, \Xset,
\ee
where CFL in Eq. \eqnref{eq:time_step} is the Courant-Friedrich-Lewi number \cite{Hirsch_book} and $\Dtc$ is the collision time-step. Equation \eqnref{eq:time_step} can be derived by means of an entropy dissipation analysis and is strictly valid for a model Boltzmann equation with a BGK collision operator \cite{Mieussens_JCP}. However, the use of Eq. \eqnref{eq:time_step} for the evaluation of the time-step did not lead to particular problems while performing the calculations presented in the paper.   

As alternative to the Forward Euler method, multi-stage schemes (such as Runge-Kutta methods \cite{Hirsch_book}) could be considered for the time-integration of Eq. \eqnref{eq:adv2}. Boundary conditions are applied through ghost cells \cite{Hirsch_book}.
\subsection{Algorithm for the evaluation of the collision operators}\label{sec:Coll_alg}
In order to evaluate the elastic and inelastic collision operators (Eq. \eqnref{eq:Qsumin}) on the discrete velocity nodes given by Eq. \eqnref{eq:vk}, the following algorithm is proposed. For the elastic collision $i+j=i+j$, the partial elastic collision operator $\Qel$ is computed as follows:
\begin{enumerate}
  \item Compute the Fourier transform of the velocity distribution function of the species $i$ and $j$: 
       $$
        {\hat{f}}_{i,j}(\fa) = \mathcal{F}(f_{i,j}(\Vi))\, \rightarrow \, O(N^{\, 3}_v \log \Nv).
       $$
  \item For $N^{\, 3}_v$ discrete Fourier velocity nodes compute the Fourier transform of the partial elastic collision operator by means of the weighted convolution in Fourier space (Eq. \eqnref{eq:FT_theor}):
        $$
        \FQel (\fa) = \int \!\! \Fvdfi \left(\fa - \fb \right) \, \Fvdfj (\fb) \, \FFGel \left(\fa,\fb \right) \, d\fb \, \rightarrow \, O(N^{\, 6}_v). \nonumber  
        $$
  \item Compute the inverse Fourier transform of the partial elastic collision operator: 
        $$
         {\tilde{Q}}_{\, ij}(\Vi) = \mathcal{F}^{-1}(\FQel (\fa))\, \rightarrow \, O(N^{\, 3}_v \log \Nv).
        $$
  \item For $N^{\, 3}_v$ discrete velocity nodes enforce conservation through the solution of a constrained optimization problem:
        $$
        Q_{\, ij}(\Vi) = \mathrm{Opt}({\tilde{Q}}_{\, ij}(\Vi))\, \rightarrow \, O(N^{\, 3}_v).\nonumber 
        $$
\end{enumerate}
The modification of the above procedure in the case of an inelastic collision is straightforward and can be deduced from Eq. \eqnref{eq:FT_theor2}. The global cost of the algorithm is $O(N^{\, 6}_v)$ (per partial collision operator) and the last step is performed in order to ensure conservation of mass, momentum and energy during collisions as stated in Eqs. \eqnref{eq:Qel_inv} (Eq. \eqnref{eq:Qin_inv} for an inelastic collision). This approach was originally proposed and formulated by Gamba \textit{et al} \cite{Gamba1,Gamba2} for the case of a pure gas without internal energy. In the present work, an extension of the original method to a multi-energy level gas is proposed. Due to the existence of separate sets of collisional invariants (elastic and inelastic), the conservation of mass, momentum and energy during collisions is enforced through the solution of two separate constrained optimization problems:
\begin{enumerate}
\item Elastic collisions:
      \be \label{eq:pel}
        \Pel = \left\{ \min\left( \sumijns \!\!\! \left|\QtVel - \QVel \right|^{\, 2}\right), \sumijns \!\!\! \Cel \, \QVel  = \zeroel \right\}.
      \ee
\item Inelastic collisions:
      \be \label{eq:pin}
        \Pin = \left\{ \min \left(\sumijhlns \!\!\!\!\!\!\! \left|\QtVin -  \QVin \right|^{\, 2} \right), \sumijhlns \!\!\!\!\!\! \Cin \, \QVin  = \zeroin \right\}.
      \ee 
\end{enumerate}
In Eqs. \eqnref{eq:pel} - \eqnref{eq:pin}, the vectors $\QVel$, $\QtVel$, $\QVin$ and $\QtVin$ store the values of the partial collision operators $\Qel$ and $\smash{\Qin}$, respectively, on the discrete velocity nodes given by Eq. \eqnref{eq:vk} (the tilde symbol is used to indicate the values obtained after the inversion of the Fourier transform that do not satisfy conservation). In the same equations, the constraints imposed represent the conservation requirements the collision operators must satisfy (Eqs. \eqnref{eq:Qel_inv} and \eqnref{eq:Qin_inv}). In view of the discretization introduced for the velocity space, this operation is realized at discrete level through multiplication with the elastic and inelastic integration matrices ($\Cel$ and $\Cin$, respectively). The columns of these matrices are precisely given by the elastic and inelastic collisional invariants (Eqs. \eqnref{eq:coll_inv_el1} - \eqnref{eq:coll_inv_el2} and Eqs. \eqnref{eq:coll_inv_in1} - \eqnref{eq:coll_inv_in2}, respectively) evaluated at the discrete velocity nodes given by Eq. \eqnref{eq:vk}. Hence, for the columns associated to the discrete velocity node $\vk$ one has:
\begin{IEEEeqnarray}{rCl}
\left(\Cel\right)_{\, \kvec} & = & \Dvc \, \wk \left[\begin{array}{ccc} \mi \, \deltaib & \mi \, \vk & \half \mi \, \vks \end{array}\right]^{\mathrm{T}}, \label{eq:Cel} \\
\left(\Cin\right)_{\, \kvec} & = & \Dvc \, \wk \left[\begin{array}{ccc} \mi & \mi \, \vk & \half \mi \, \vks + \Einti \end{array}\right]^{\mathrm{T}}, \quad \iset, \, \kvec \, \in \, \Vsetc. \label{eq:Cin}
\end{IEEEeqnarray}
In Eqs. \eqnref{eq:Cel} - \eqnref{eq:Cin}, the quantity $\wk = \wkx \, \wky \, \wkz$ is the global integration weight associated to the velocity node $\vk$, $\smash{\vks = v^{\, 2}_{\, \kx} + v^{\, 2}_{\, \ky} + v^{\, 2}_{\, \kz}}$ and $\deltaib$ is a vector made of $\Ns$ components whose $j^{\, \rm th}$ component is $\delta_{ij}$. 
\newtheorem{Opt_el}[FT]{Proposition}
\begin{Opt_el}
The solution of the constrained optimization problem $\Pel$ for elastic collisions (Eq. \eqnref{eq:pel}) is:
\be 
\QVel = \QtVel - \CelT \, {\Csumel}^{-1} \, \Qsumel,\quad \ijset. \label{eq:sol_el}
\ee
In Eq. \eqnref{eq:sol_el}, the symbol $\Trm$ is used to indicate the transpose operator while the matrix $\Csumel$ and the vector $\Qsumel$ are, respectively, defined as:
\begin{IEEEeqnarray}{rCl}
\Csumel & = & \Ns \!\! \sumins \! \Cel \, \CelT,\\
\Qsumel & = & \!\!\! \sumijns \!\!\! \Cel \, \QtVel. 
\end{IEEEeqnarray}
\end{Opt_el}
\begin{proof}
The Lagrangian associated to the constrained optimization problem $\Pel$ in Eq. \eqnref{eq:pel} is:
\be \label{eq:lagr}
\Lagrel = \!\!\! \sumijns \!\!\! \left|\QtVel - \QVel \right|^{\, 2} + \,\, \lamelT \!\! \sumijns \!\!\! \Cel \, \QVel.
\ee
The vector $\lamel$ in Eq. \eqnref{eq:lagr} is the Lagrange multiplier vector and has $\Ns + 4$ components. The solution of the problem $\Pel$ is given by the stationary points of the Lagrangian $\Lagrel$ (Eq. \eqnref{eq:lagr}). These are found by imposing:
\begin{IEEEeqnarray}{rCl}
\fr{\pa \, \Lagrel}{\pa \, \QVel} & = & \zeroel, \quad \ijset, \label{eq:lag1} \\
\fr{\pa \, \Lagrel}{\pa \, \lamel} & = & \zeroel. \label{eq:lag3}
\end{IEEEeqnarray}  
The application of Eqs. \eqnref{eq:lag1} - \eqnref{eq:lag3} leads to:
\begin{IEEEeqnarray}{rCl}
\QVel & = & \QtVel - \half \CelT \, \lamel, \quad \ijset, \label{eq:lag4} \\
\zeroel & = & \!\!\! \sumijns \!\!\! \Cel \, \QVel. \label{eq:lag6} 
\end{IEEEeqnarray}
The left multiplication of Eq. \eqnref{eq:lag4} by the matrix $\Cel$ and the sum of the result obtained over all elastic collisions gives (after some algebra):
\be \label{eq:lag12}
\lamel = 2 \left[ \Ns \!\! \sumins \! \Cel \, \CelT\right]^{-1} \left( \sumijns \!\!\! \Cel \, \QtVel \right).
\ee
The substitution of Eq. \eqnref{eq:lag12} in Eq. \eqnref{eq:lag4} gives the thesis.
\end{proof}
\noindent 
Equation \eqnref{eq:sol_el} reduces to the original result obtained by Gamba \textit{et al} \cite{Gamba1,Gamba2} for the case of a pure gas without internal energy:
\be \label{eq:sol_el_orig}
\mbf{Q} = \mbf{\tilde{Q}} - {\mbf{C}^{\, \elrm}}^{\mrm{T}} \left(\mbf{C}^{\, \elrm} \, {\mbf{C}^{\, \elrm}}^{\mrm{T}}\right)^{-1} \!\! \mbf{C}^{\, \elrm} \, \mbf{\tilde{Q}},
\ee
where the elastic integration matrix $\mbf{C}^{\, \elrm}$ in Eq. \eqnref{eq:sol_el_orig} is obtained from Eq. \eqnref{eq:Cel} when $\Ns = 1$:
\be 
\mbf{C}^{\, \elrm} = \Dvc \, \wk \left[\begin{array}{ccc} m & m \, \vk & \half m \, \vks  \end{array}\right]^{\mathrm{T}}, \quad \kvec \, \in \, \Vsetc.
\ee
\newtheorem{Opt_in}[FT]{Proposition}
\begin{Opt_in}
The solution of the constrained optimization problems $\Pin$ for inelastic collisions (Eq. \eqnref{eq:pin}) is:
\be
\QVin = \QtVin - \CinT \, {\Csumin}^{-1} \, \Qsumin ,\quad \ijhlsetC. \label{eq:sol_in}
\ee
In Eq. \eqnref{eq:sol_in}, the matrix $\Csumin$ and the vector $\Qsumin$ are, respectively, defined as:
\begin{IEEEeqnarray}{rCl}
\Csumin & = & \Nin \!\!  \sumins \! \Cin \, \CinT ,\\
\Qsumin & = & \!\!\!\!\!\! \sumijhlns \!\!\!\!\!\! \Cin \, \QtVin. 
\end{IEEEeqnarray}
\end{Opt_in}
\begin{proof}
The Lagrangian associated to the constrained optimization problem $\Pin$ in Eq. \eqnref{eq:pin} is:
\be \label{eq:lagrin}
\Lagrin = \!\!\!\!\! \sumijhlns \!\!\!\!\!\!\! \left|\QtVin -  \QVin \right|^{2} + \,\, \laminT \!\!\!\!\! \sumijhlns \!\!\!\!\!\! \Cin \, \QVin.
\ee
The vector $\lamin$ in Eq. \eqnref{eq:lagrin} is the Lagrange multiplier vector and has $5$ components. The solution of the problem $\Pin$ is given by the stationary points of the Lagrangian $\Lagrin$ (Eq. \eqnref{eq:lagrin}). These are found by imposing:
\begin{IEEEeqnarray}{rCl}
\fr{\pa \, \Lagrin}{\pa \, \QVin} & = & \zeroin, \quad \ijhlsetC, \label{eq:lag1in} \\
\fr{\pa \, \Lagrin}{\pa \, \lamin} & = & \zeroin. \label{eq:lag3in}
\end{IEEEeqnarray}  
The application of Eqs. \eqnref{eq:lag1in} - \eqnref{eq:lag3in} leads to:
\begin{IEEEeqnarray}{rCl}
\QVin & = & \QtVin - \half \CinT \, \lamin, \quad \ijhlsetC, \label{eq:lag4in} \\
\zeroin & = & \!\!\!\!\! \sumijhlns \!\!\!\!\!\! \Cin \, \QVin. \label{eq:lag6in} 
\end{IEEEeqnarray}
The left multiplication of Eq. \eqnref{eq:lag4in} by the matrix $\Cin$ and the sum of the result obtained over all the inelastic collisional processes given by the set $\Cinset$ gives (after some algebra):
\be \label{eq:lag12in}
\lamin = 2 \left[\Nin \!\! \sumins \! \Cin \, \CinT \,\right]^{-1} \left( \sumijhlns \!\!\!\!\!\!\! \Cin \, \QtVin \right).
\ee
The substitution of Eq. \eqnref{eq:lag12in} in Eq. \eqnref{eq:lag4in} gives the thesis.
\end{proof}
\newpage
\section{Computational results}\label{sec:res}
The numerical method presented in detail in Sect. \ref{sec:num} has been implemented in a parallel C code (Boltzmann Equation Spectral-Lagrangian Solver - BESS in what follows). Parallelization is performed by means of the OpenMP library \cite{OpenMP_book}. The FFTW \cite{fftw,fftw_web} (Fastest-Fourier-Transform in the West) and the GSL \cite{gsl} (GNU-Scientific Library) packages have been used, respectively, for the evaluation of FFTs (and inverse FFTs) and vector/matrix manipulation. 

Both space homogeneous and space in-homogeneous benchmarks have been considered. The former consist in studying the evolution towards equilibrium of isochoric systems initially set in a non-equilibrium state, while the latter consist in computing the steady-state flow across normal shock waves. In all the cases, macroscopic moments (given in Sects. \ref{sec:mom_cons} and \ref{sec:neq}) have been computed and compared with the DSMC results obtained by Torres \cite{Erik_RGD28}. The numerical approximation of the integrals defining the macroscopic moments are given in \ref{app:mom}.
\subsection{Isochoric equilibrium relaxation of a Ne-Ar mixture}\label{sec:Ne_Ar_Hom}
The system under investigation is a binary mixture of Neon and Argon. The electronic energy of the atoms is assumed to be negligible. Only elastic collisions are accounted for. The hard-sphere collision model \cite{Bird_book} is used for the differential cross-section, $\sigel = (\diami + \diamj)^{\, 2}/16$ with $\diami$ and $\diamj$ being, respectively, the diameters of the species $i$ and $j$. The  system is initially set in a non-equilibrium state where both species follow a Maxwell-Boltzmann velocity distribution function (Eq. \eqnref{eq:MB}) with zero hydrodynamic velocity at different temperatures. The numerical values of the species diameter and mass (taken from \cite{Bird_book}) are reported in Table \ref{tab:Ne_Ar_dm} together with the values of the species density and initial temperature. The mixture temperature corresponding to the conditions provided in Table \ref{tab:Ne_Ar_dm} is $333.63 \, \rm K$.

The velocity space is discretized by adopting the values of $\Nv = 24$ nodes and $\Lv = 3000 \, \rm m/s$ (see Sect. \ref{sec:phase}). The collision time-step $\Dtc$ is set to $1 \times 10^{\, -9} \, \rm s$ in order to have a value lower than the mean collision time. The number of partial elastic collision operators to be evaluated at each time-step is equal to 4. The simulation is stopped after 300 time-steps. The CPU time required is approximately 3 minutes when using 4 threads.
\begin{table}[!htbf]
\begin{center}
\begin{tabular}{lllll}
$i$ & $\mi \, [\mrm{kg}]$ & $\diami \, [\mrm{m}]$ & $\rhoi \, [\mrm{kg/m^{3}}]$ & $\Ti \, [\mrm{K}]$ \\
\hline
Ne & $3.35 \times 10^{\, -26}$ & $2.77 \times 10^{\, -10}$ &  $5 \times 10^{\, -3}$  &  300 \\
Ar    & $6.63 \times 10^{\, -26}$ & $4.17 \times 10^{\, -10}$ &  $2 \times 10^{\, -3}$  &  500  \\
\hline
\end{tabular}
\caption{Species mass, diameter, density and initial temperature.}
\label{tab:Ne_Ar_dm}
\end{center}
\end{table}
\begin{figure}[!htbf]
\centering
\subfloat[Density.]{
\psfrag{time}[c][c][0.8]{$\tleg$}
\psfrag{dens}[c][c][0.8]{$\rholeg$}
\psfrag{a}[l][l][0.7]{$\rm Mixture$}
\psfrag{b}[l][l][0.7]{$\Nerm$}
\psfrag{c}[l][l][0.7]{$\Arrm$}
{\includegraphics[scale=0.285,keepaspectratio]{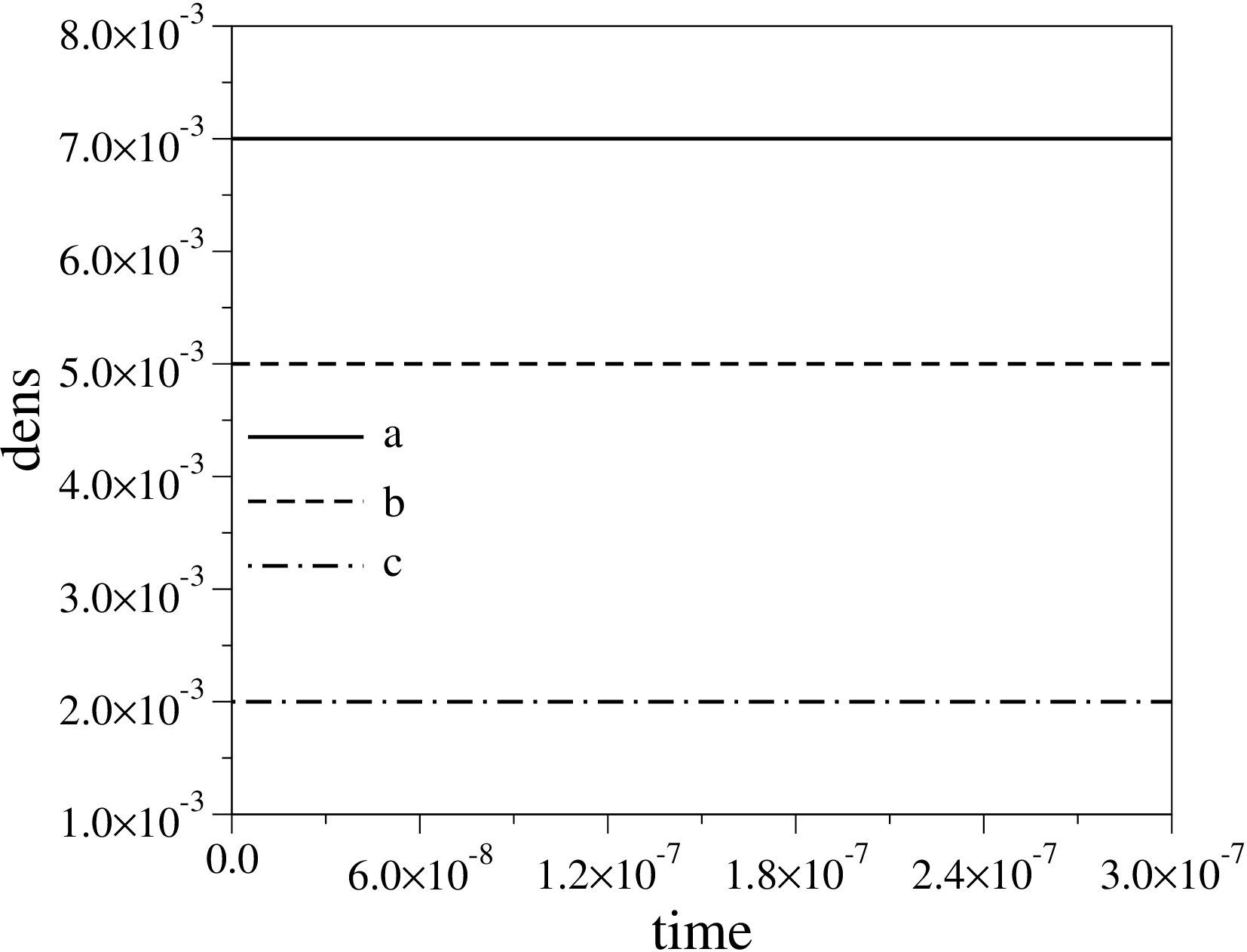}}} \,  
\subfloat[Temperature.]{
\psfrag{time}[c][c][0.8]{$\tleg$}
\psfrag{temp}[c][c][0.8]{$\Tleg$}
{\includegraphics[scale=0.285,keepaspectratio]{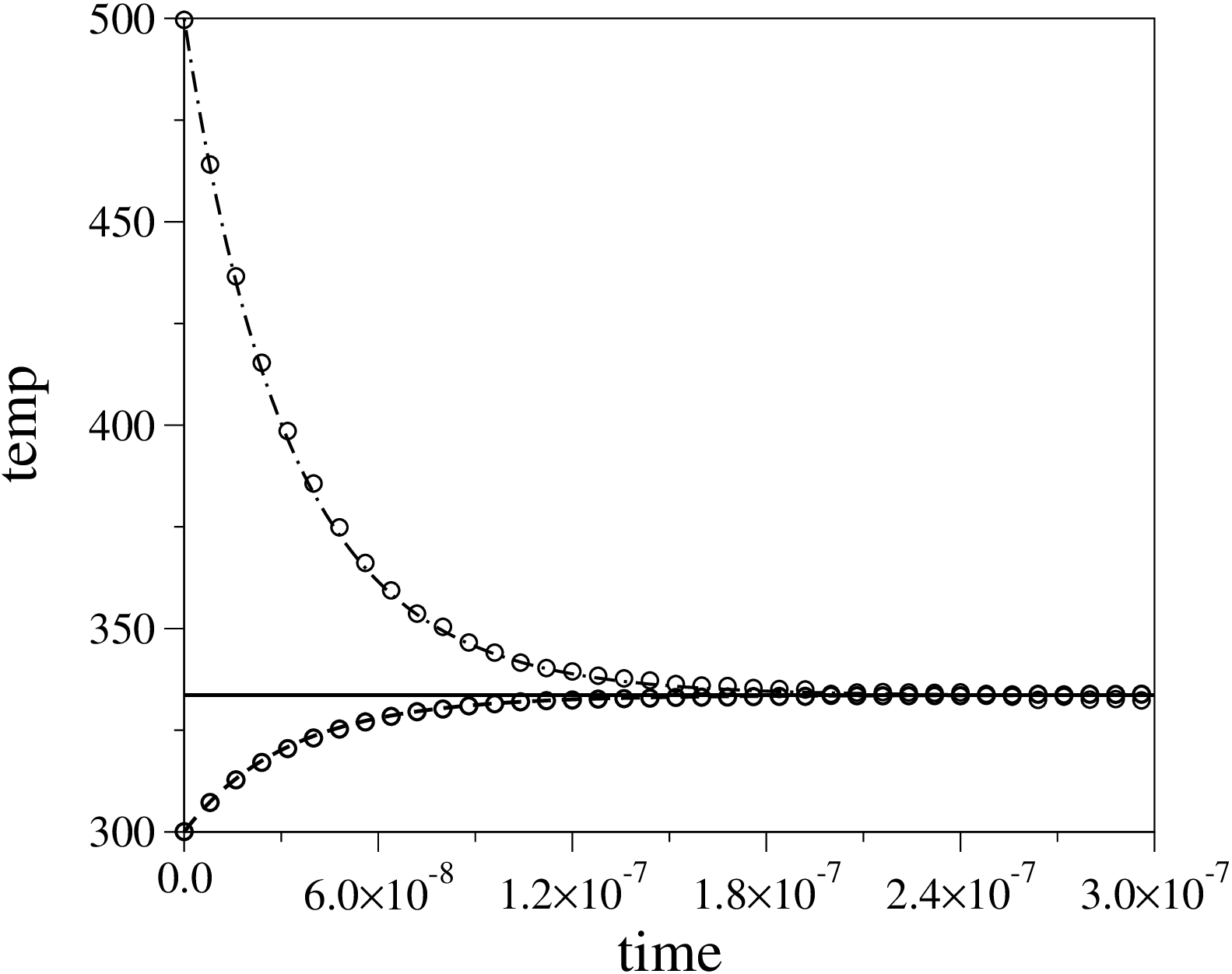}}} \,
\caption{Isochoric equilibrium relaxation of a Ne-Ar mixture: time-evolution of the species and mixture density and temperature (lines BESS - symbols DSMC).}\label{fig:box_Ne_Ar_mom} 
\end{figure}  

\newpage
Once the simulation is started, collisions bring the system from its initial non-equilibrium condition to the final equilibrium state. Since the system is isochoric and no external mass, momentum and energy sources are present, the following statements hold:
\bi 
\item The density of each species is constant and maintains its initial value. The same can be said for the mixture density.  
\item The species and mixture hydrodynamic velocity is constant and maintains its initial value (zero).
\item The mixture temperature is constant and maintains its initial value. On the other hand, the temperature of each species experiences variation and approaches the mixture temperature value at equilibrium.
\ei  
The foregoing are a direct consequence of mass, momentum and energy conservation during collisions and should be obtained as a result if the numerical method used for solving the Boltzmann equation is conservative. In order to assess that, the time-evolution of the species and mixture density and temperature is monitored (see Fig. \ref{fig:box_Ne_Ar_mom}). Both the mass and mixture densities remain constant and do not show any variation. The same is valid for the mixture temperature, while the species temperatures evolves towards the correct equilibrium value. The species and mixture hydrodynamic velocities retain their initial values (zero) and are not shown in Fig. \ref{fig:box_Ne_Ar_mom}. From the analysis of the results shown in Fig. \ref{fig:box_Ne_Ar_mom}, one can conclude that the proposed extension of the original spectral-Lagrangian method \cite{Gamba1,Gamba2} to a mixture of monatomic gases without internal energy enables to respect the requirements stated in Eq. \eqnref{eq:Qel_inv}. The agreement with the DSMC results is excellent.   

Figure \ref{fig:box_Ne_Ar_vdf} shows the time-evolution of the $\vx$ axis component of the species velocity distribution functions (the $\vy$ and $\vz$ axis components are not shown because they are practically identical to the $\vx$ component). The results obtained show that the evolution towards the equilibrium state occurs through sequences of Maxwell-Boltzmann velocity distribution functions.
\begin{figure}[!htbf]
\centering
\subfloat[$\Nerm$.]{
\psfrag{v}[c][c][0.8]{$\vxleg$}
\psfrag{vdf}[c][c][0.8]{$\vdfvxlegHom$}
\psfrag{a}[l][l][0.6]{$t = 5.0 \times 10^{\, -9} \, \srm$}
\psfrag{b}[l][l][0.6]{$t = 1.4 \times 10^{\, -8} \, \srm$}
\psfrag{c}[l][l][0.6]{$t = 2.4 \times 10^{\, -8} \, \srm$}
\psfrag{d}[l][l][0.6]{$t = 5.0 \times 10^{\, -8} \, \srm$}
\psfrag{e}[l][l][0.6]{$t = 1.5 \times 10^{\, -7} \, \srm$}
{\includegraphics[scale=0.28,keepaspectratio]{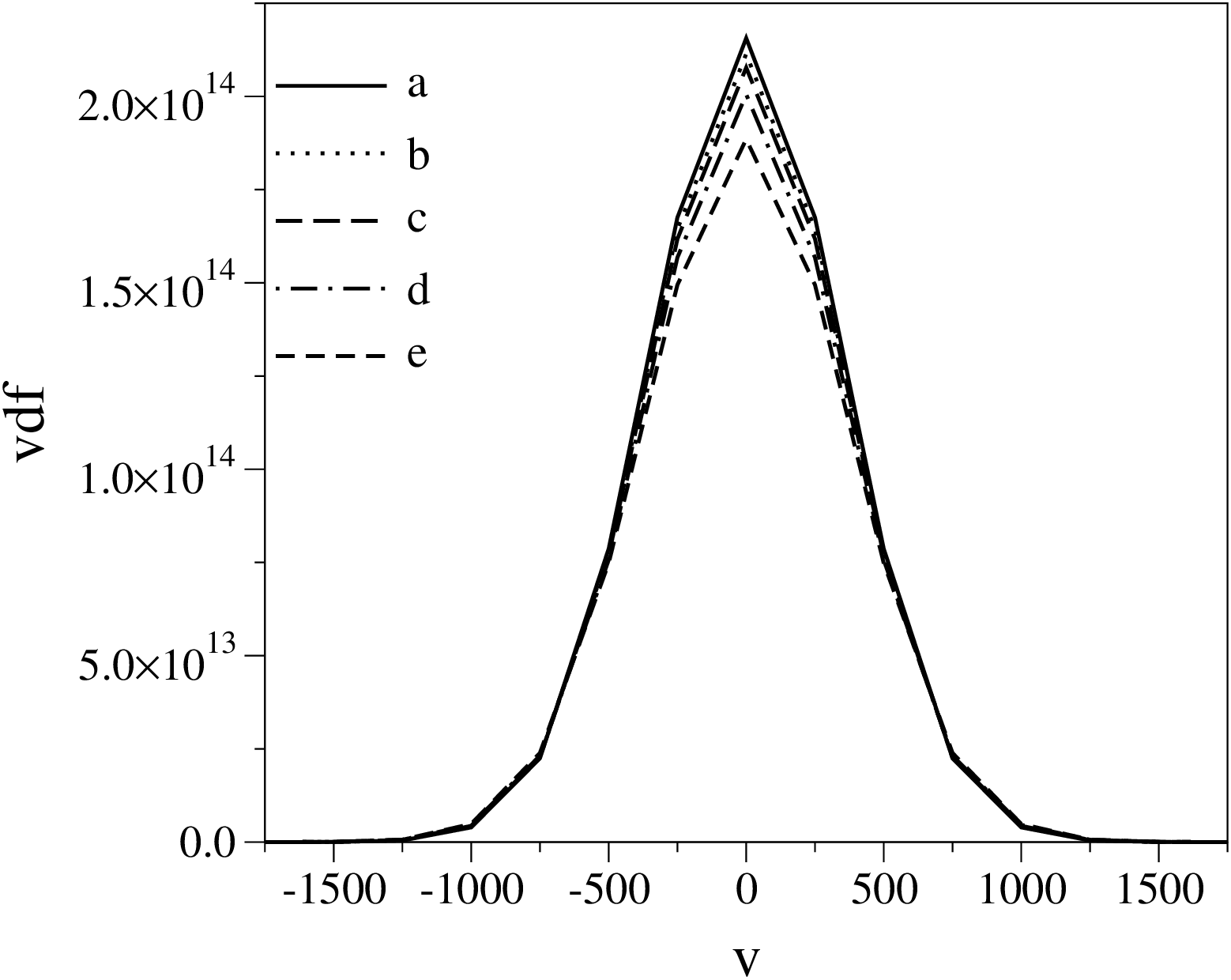}}} \,  
\subfloat[$\Arrm$.]{
\psfrag{v}[c][c][0.8]{$\vxleg$}
\psfrag{vdf}[c][c][0.8]{$\vdfvxlegHom$}
{\includegraphics[scale=0.28,keepaspectratio]{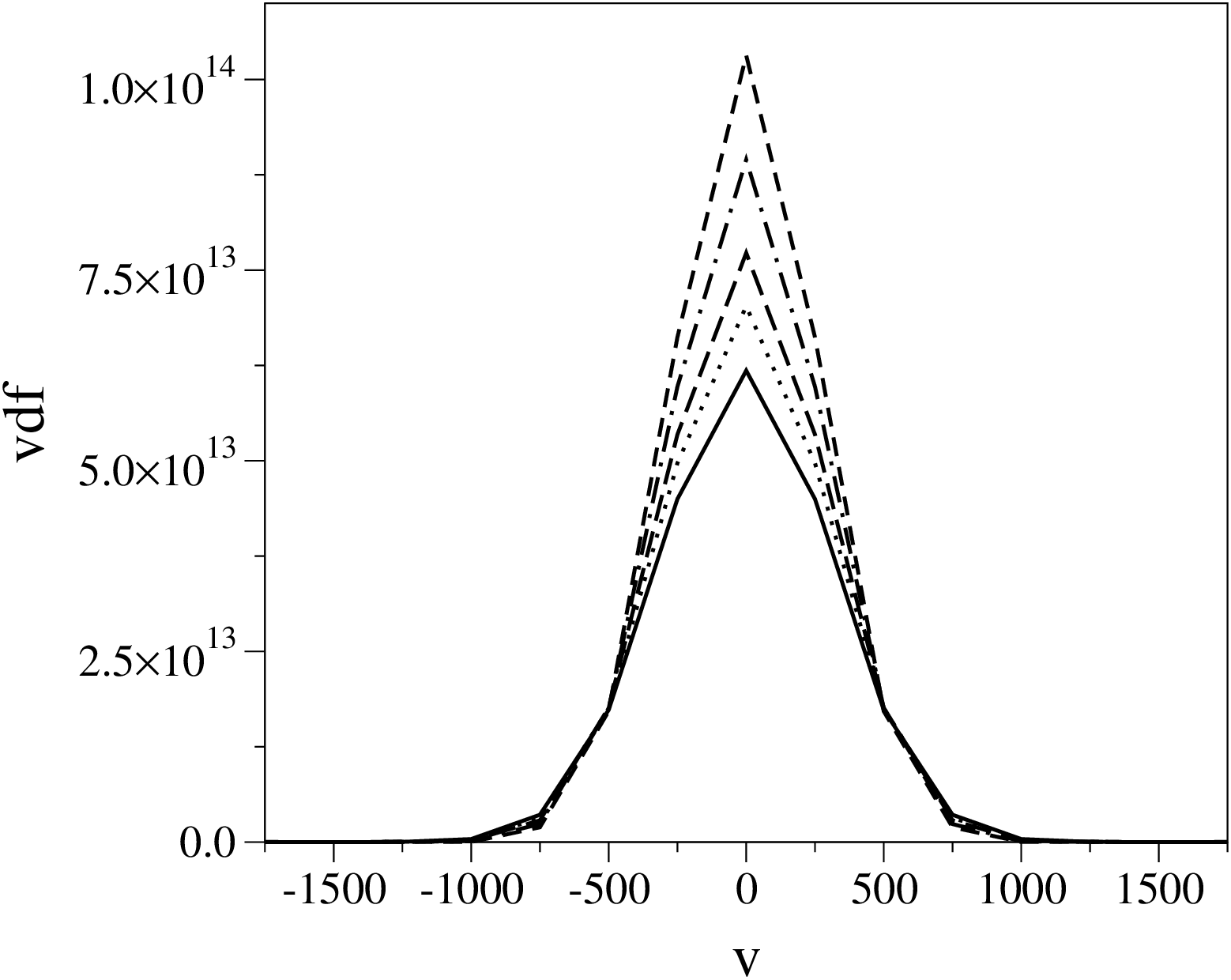}}} \,
\caption{Isochoric equilibrium relaxation of a Ne-Ar mixture: time-evolution of the $\vx$ axis component of the species velocity distribution function.}\label{fig:box_Ne_Ar_vdf} 
\end{figure} 
\newpage 
\subsection{Isochoric equilibrium relaxation of a multi-energy level gas}\label{sec:Ar_star_Hom}
The system under investigation consists of a monatomic gas with 5 internal energy levels whose values for degeneracy and energy (taken from \cite{Anderson_RGD16}) are given in Table \ref{tab:en_lev}. The mass of the gas particles $m$ is equal to that of Argon (the value is equal to that used in Sect. \ref{sec:Ne_Ar_Hom}) and its diameter $d$ is $3.0 \times 10^{\, -10} \, \rm m$.
\begin{table}[!htbf]
\begin{center}
\begin{tabular}{lll}
$i$ & $\gi$ & $\Einti \, [\mathrm{J}]$ \\
\hline
1 & 1 & 0.0 \\
2 & 1 & $8.30 \times 10^{\, -21}$ \\
3 & 1 & $1.66 \times 10^{\, -20}$ \\
4 & 1 & $2.50 \times 10^{\, -20}$ \\
5 & 1 & $3.30 \times 10^{\, -20}$ \\ 
\hline
\end{tabular}
\caption{Level degeneracy and energy.}
\label{tab:en_lev}
\end{center}
\end{table}  

Both elastic and inelastic collisions are allowed to occur. For the evaluation of the related cross-sections, the model proposed by Anderson \cite{Anderson_RGD16} is considered. According to this model, the differential cross-section associated to the collision $\inter$ is written as a product between a hard-sphere differential cross-section $d^{\, 2}/4$ and a transition probability $\smash{\tprob}$, that is, $\smash{\sigin = \tprob \, d^{\, 2}/4}$. The transition probability $\smash{\tprob}$ only depends on the pre-collisional relative velocity magnitude $u$ and has the following expression:
\be \label{eq:And_prob}
\tprob = \fr{\max \left[\gh \, \gl \, \left(\mu \us - 2 \, \DEijhl\right),0\right] }{\displaystyle \sum_{m,n \, \in \, \Sset} \!\!\! \max\left[ g_m \, \gn \, \left(\mu \us - 2 \, \DEijmn\right) ,0\right]}, \quad \ijhlset,
\ee   
where the reduced mass of the colliding species $\mu$ in Eq. \eqnref{eq:And_prob} is equal to $m/2$ for the present simulation. Notice that Eq. \eqnref{eq:And_prob} comprises also the case of elastic collisions.

The initial state of the system corresponds to a partial equilibrium condition. The velocity distribution functions of all levels (species) is a two temperature (translational $T$ and internal $\Tint$) Maxwell-Boltzmann velocity distribution function (Eq. \eqnref{eq:MB}) with zero bulk velocity. This is obtained by assuming that the level densities appearing in Eq. \eqnref{eq:MB} are given by the Boltzmann distribution law (Eq. \eqnref{eq:Boltz_distr}) at the internal temperature $\Tint$. 

The gas has a density of $1 \, \rm kg/m^3$. The initial values of the translational and internal temperatures are $1000 \, \rm K$ and $100 \, \rm K$, respectively. The initial condition of the system approximates the state of a gas immediately behind a normal shock wave when this is treated as a discontinuity.

The velocity space is discretized by adopting the values of $\Nv = 16$ nodes and $\Lv = 3000 \, \rm m/s$ (see Sect. \ref{sec:phase}). The collision time-step $\Dtc$ is set to $1 \times 10^{\, -12} \, \rm s$ in order to have a value lower than the mean collision time (based on a hard-sphere collision model). The number of partial collision operators to be evaluated at each time-step is equal to 625 (25 elastic and 600 inelastic). The simulation is stopped after 2500 time-steps. The CPU time required is approximately 2 hours when using 12 threads.

As for the case studied in Sect. \ref{sec:Ne_Ar_Hom}, when the simulation is started, collisions bring the system to equilibrium. However, due to the presence of inelastic collisions, some differences arise in the time-evolution:
\bi 
\item The mixture density is constant and maintains its initial value. On the other hand, the density of each level changes in time and evolves from the initial non-equilibrium condition to its final equilibrium value.    
\item The level and mixture hydrodynamic velocity is constant and maintains its initial value (zero).
\item The mixture temperature changes in time and evolves from the initial non-equilibrium condition to its final equilibrium value. 
\ei  
 The value of the temperature at equilibrium can be computed from the energy balance between the initial and the final equilibrium state. For the present simulation, the value of $723.5 \, \rm K$ is obtained. Once the equilibrium temperature determined, it is possible to compute the equilibrium values of the level densities by means of the Boltzmann distribution law given in Eq. \eqnref{eq:Boltz_distr}.

In order to assess the conservation properties of the proposed spectral-Lagrangian method for the case of a multi-energy level gas, the time-evolution of the density of each level and the translational and internal temperatures are monitored (see Fig. \ref{fig:box_Ar_star}). The time-evolution of the level densities and temperatures given in Fig. \ref{fig:box_Ar_star} confirms the previous considerations regarding the behavior of the system. In particular, the population of the ground state (first level) decreases while that of the upper states increase. The translational temperature decreases till the equilibrium value is not reached. The opposite behavior (as expected) is observed for the internal temperature. The former demonstrates the existence of a net macroscopic energy transfer from the translational to the internal degree of freedom of the gas. The level and mixture hydrodynamic velocities retain their initial values (zero) and are not shown in Fig. \ref{fig:box_Ar_star}. The agreement with the DSMC solution (also shown in Fig. \ref{fig:box_Ar_star}) is excellent.         
\begin{figure}[!htbf]
\centering
\subfloat[Level density.]{
\psfrag{time}[c][c][0.9]{$\tleg$}
\psfrag{dens}[c][c][0.9]{$\rho \, [\mathrm{kg/m^3}]$}
\psfrag{a}[c][c][0.8]{$\rho_1$}
\psfrag{b}[c][c][0.8]{$\rho_2$}
\psfrag{c}[c][c][0.8]{$\rho_3$}
\psfrag{d}[c][c][0.8]{$\rho_4$}
\psfrag{e}[c][c][0.8]{$\rho_5$}
{\includegraphics[scale=0.285,keepaspectratio]{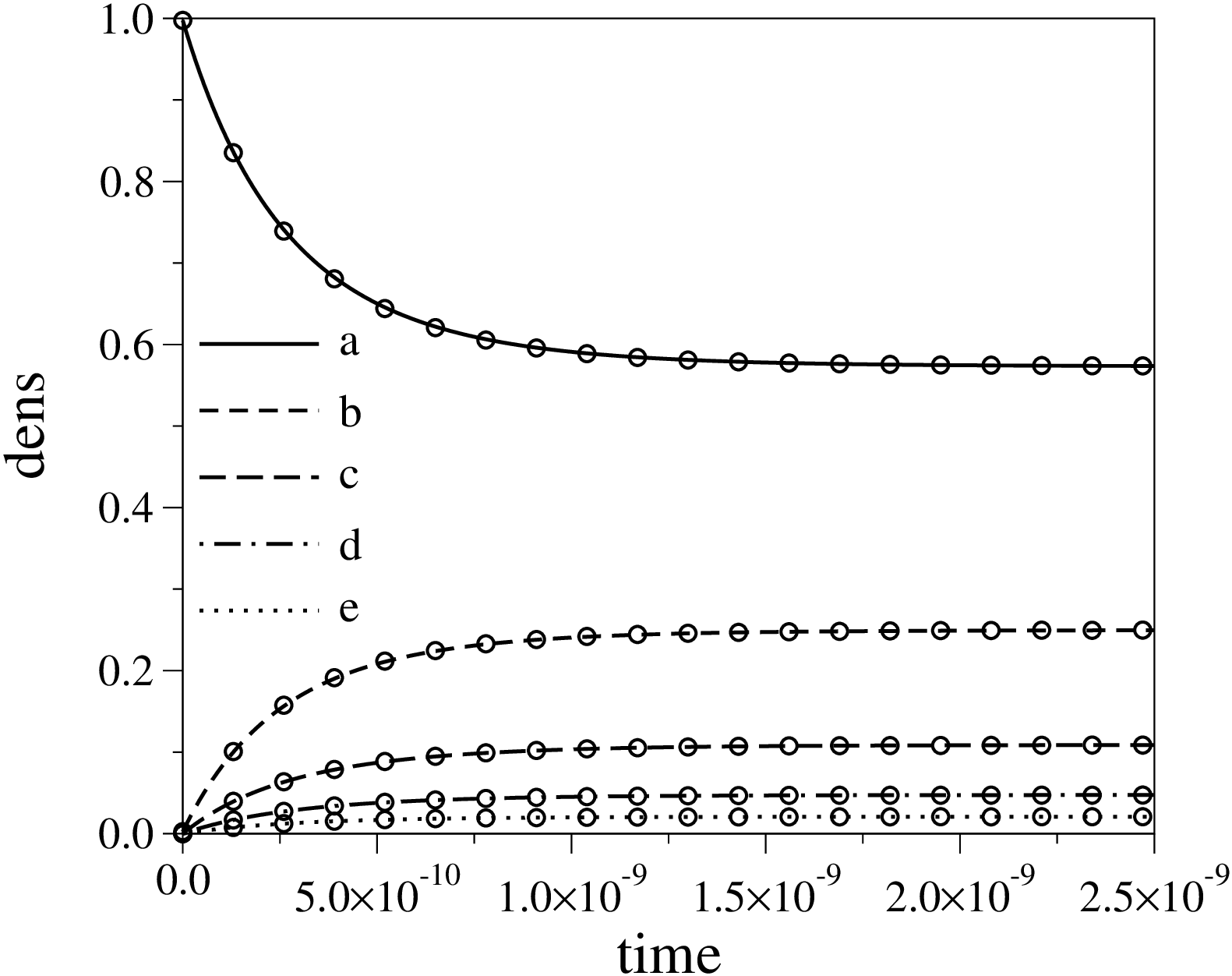}}} \,
\subfloat[Translation and internal temperature.]{
\psfrag{time}[c][c][0.9]{$\tleg$}
\psfrag{temp}[c][c][0.9]{$T, \Tint \, [\mathrm{K}]$}
\psfrag{a}[l][l][0.8]{$T$}
\psfrag{b}[l][l][0.8]{$\Tint$}
{\includegraphics[scale=0.285,keepaspectratio]{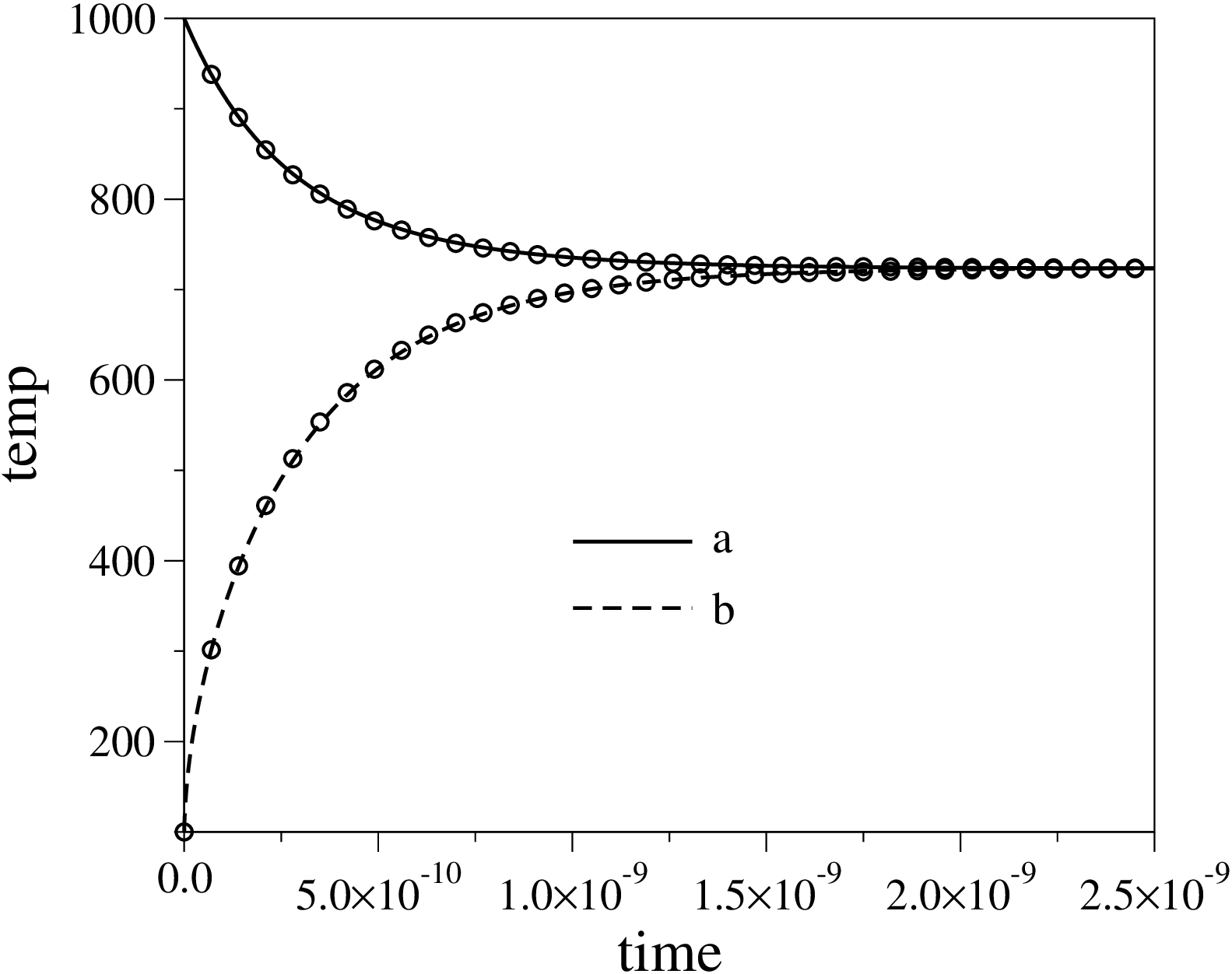}}} \,
\caption{Isochoric equilibrium relaxation of a multi-energy level gas: time-evolution of the level density, translational temperature and internal temperature (lines BESS - symbols DSMC).}\label{fig:box_Ar_star}
\end{figure}      
\begin{table}[!htbf]
\begin{center}
\begin{tabular}{llllllll}
     & $T \, [\mrm{K}] $ & $\rho_1 \, [\mrm{kg/m^3}] $ & $\rho_2 \, [\mrm{kg/m^3}] $ & $\rho_3 \, [\mrm{kg/m^3}] $ & $\rho_4 \, [\mrm{kg/m^3}] $ & $\rho_5 \, [\mrm{kg/m^3}] $  \\
\hline
BESS & 723.4029 & 0.573 & 0.245 & 0.1088 & 0.0474 & 0.02069 \\
eq  & 723.543 & 0.573 & 0.245 & 0.1089 & 0.0474 & 0.02064 \\
\hline
\end{tabular}
\caption{Final values of temperature and level density (comparison between simulation and equilibrium calculation).}
\label{tab:eq_comp}
\end{center}
\end{table}

\noindent Table \ref{tab:eq_comp} compares the final values of the temperature and the level densities as obtained from the simulation with those determined by means of equilibrium calculations. The agreement between the two data sets very good. This further confirms that the proposed spectral-Lagrangian method allows for respecting the conservation requirements as stated in Eqs. \eqnref{eq:Qel_inv} and \eqnref{eq:Qin_inv} when both elastic and inelastic collisions are accounted for.     
\subsection{Flow across a normal shock wave of a Ne-Ar mixture}\label{sec:shock_Ne_Ar}
The flow across a normal shock wave of a mixture of Neon and Argon is computed by solving the space in-homogeneous Boltzmann equation in the shock wave reference frame (where the shock velocity is zero). The physical model in use (in terms of species diameter and mass, and elastic collision cross-section) is the same as that used for the space homogeneous calculations shown in Sect. \ref{sec:Ne_Ar_Hom}. A peculiar aspect of this flow is the species separation occurring within the shock wave. The latter is due to the mass difference between the two species \cite{Bird_book} with the lighter species experiencing the compression sooner than the heavier one. This fact has been confirmed by both DSMC calculations \cite{Bird_book} and experimental measurements \cite{shock_He_Ar_Exp}.  

The mixture is composed of $50\%$ of Neon and $50\%$ of Argon. The corresponding species mass fractions ($y_i = \rhoi/\rho, \, \iset$) are $0.34$ and $0.66$, respectively. The mixture free-stream ($\infty$) density, temperature and velocity are $1 \times 10^{-4} \, \rm kg/m^3$, $300 \, \rm K$ and $744 \, \rm m/s$, respectively. The latter correspond to a mixture free-stream Mach number equal to 2. Post-shock (ps) values for mixture density, velocity and temperature are computed based on the Rankine-Hugoniot jump relations \cite{Anderson_book} and are $2.29 \times 10^{\, -4} \, \rm kg/m^3$, $623.44 \, \rm K $ and $325.45 \, \rm m/s$, respectively.  

The numerical values of the parameters used for the discretization of the phase-space (Sect. \ref{sec:phase}) and the application of the time-marching method (Sect. \ref{sec:op_split}) are provided in Table \ref{tab:par_shock_Ne_Ar}. The position space is discretized by using a uniform Finite volume grid.     
\begin{table}[!htbf]
\begin{center}
\begin{tabular}{llllllll}
$\Nv$ & $\Lv \, [\mrm{m/s}] $ & $\Nx$ & $\Lmx \, [\mrm{m}]$ & $\Lpx \, [\mrm{m}]$ & $\Dtc \, [\mrm{s}] $ & CFL & Limiter  \\
\hline
22    & 3200 & 201 & $2 \times 10^{\, -2}$ & $2 \times 10^{\, -2}$ & $1 \times 10^{\, -8}$ & 0.5 & van Albada \\
\hline
\end{tabular}
\caption{Simulation parameters.}
\label{tab:par_shock_Ne_Ar}
\end{center}
\end{table}    

In the present simulation, the gas flow is directed along the positive direction of the $x$ axis of the position space. At the boundaries $x = - \Lmx$ and $x = \Lmx$, a Maxwell-Boltzmann velocity distribution function (Eq. \eqnref{eq:MB}) corresponding, respectively, to the pre and post-shock conditions is imposed for each species. The numerical solution is initialized by prescribing the pre-shock Maxwell-Boltzmann velocity distribution function in the interval $- \Lmx \leq x \leq 0$, while the post-shock Maxwell-Boltzmann velocity distribution function is used for the remaining part of the position space. The time-marching method described in Sect. \ref{sec:op_split} is then applied until the steady-state is not reached. 

In order to perform a meaningful comparison with the results obtained by means of the DSMC method, a common origin has to be determined for the numerical solutions. The latter is taken at the location where the normalized density difference $(\rho - \rhoinf)/(\rhops - \rhoinf)$ is equal to 0.5 \cite{Bird_book}.

Figure \ref{fig:shock_Ne_Ar_comp} shows the evolution across the shock wave of the species hydrodynamic velocity and parallel temperature. The results confirm, as expected, that the Neon experiences the compression before the Argon. This effect progressively disappears while the flow approaches the post-shock equilibrium state (where no species separation exists). The parallel temperature of both species does not show a monotone behavior. Instead, it reaches a maximum and then approaches the post-shock equilibrium value. This feature of the flow-field is due the distortion (along the $\vx$ axis of the velocity space) experienced by the species velocity distribution functions while the flow crosses the shock wave (see Fig. \ref{fig:shock_Ne_Ar_vdf_ev}). Notice that the peak is more pronounced for the heavier species (Argon). The comparison with the DSMC results is again very good. A further confirmation to that is provided by Fig. \ref{fig:shock_Ne_Ar_mix} showing the evolution across the shock wave of the mixture density and temperature (together with the related parallel and transverse components).  
\begin{figure}[!htbf]
\centering
\subfloat[Hydrodynamic velocity.]{
\psfrag{pos}[c][c][0.8]{$\xleg$}
\psfrag{vel}[c][c][0.8]{$\Vxleg$}
\psfrag{a}[l][c][0.75]{$\Nerm$}
\psfrag{b}[l][c][0.75]{$\Arrm$}
\psfrag{c}[l][c][0.75]{$\Nerm$}
\psfrag{d}[l][c][0.75]{$\Arrm$}
{\includegraphics[scale=0.28,keepaspectratio]{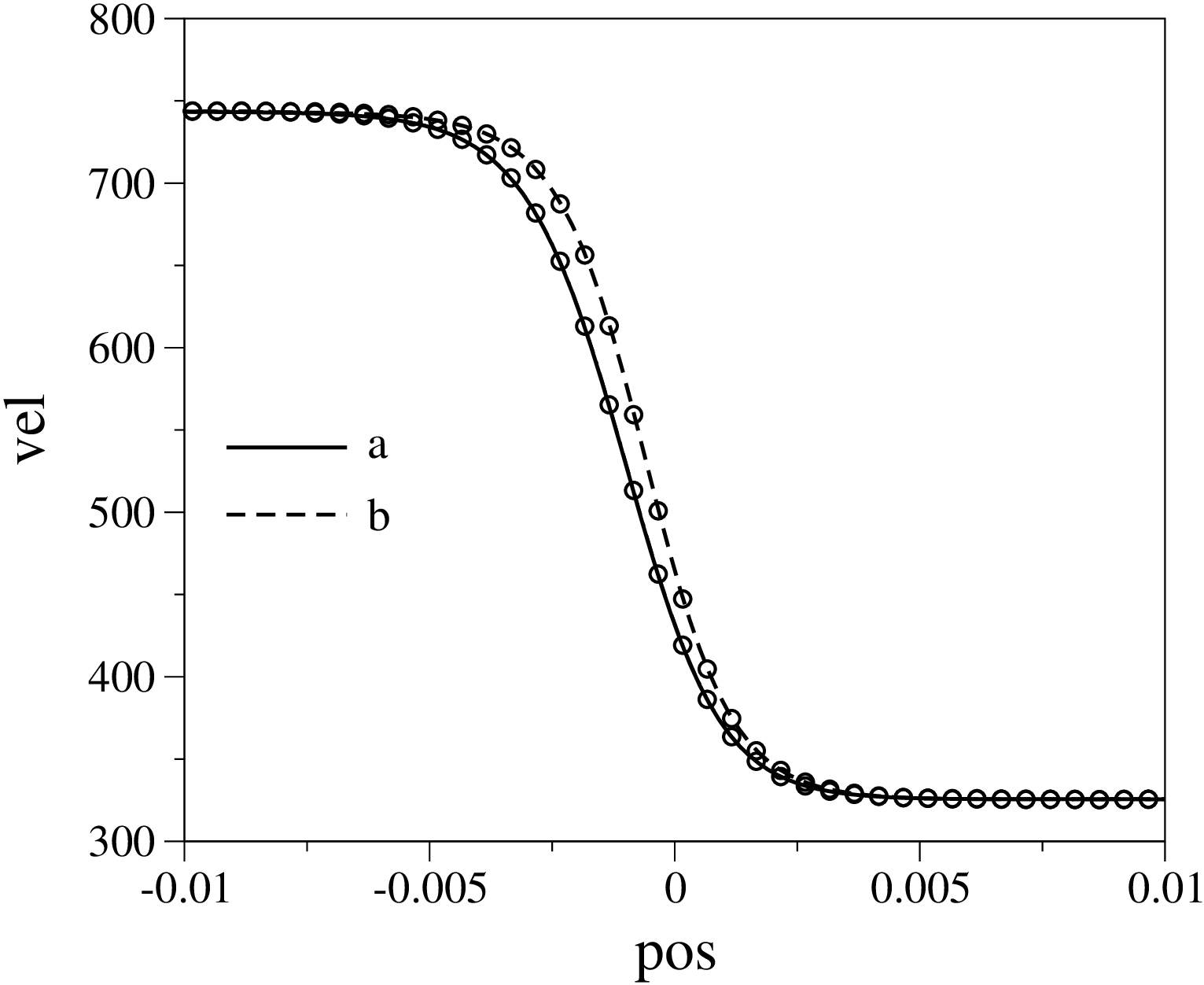}}} \,  
\subfloat[Parallel temperature.]{
\psfrag{pos}[c][c][0.8]{$\xleg$}
\psfrag{temp}[c][c][0.8]{$\Txleg$}
{\includegraphics[scale=0.28,keepaspectratio]{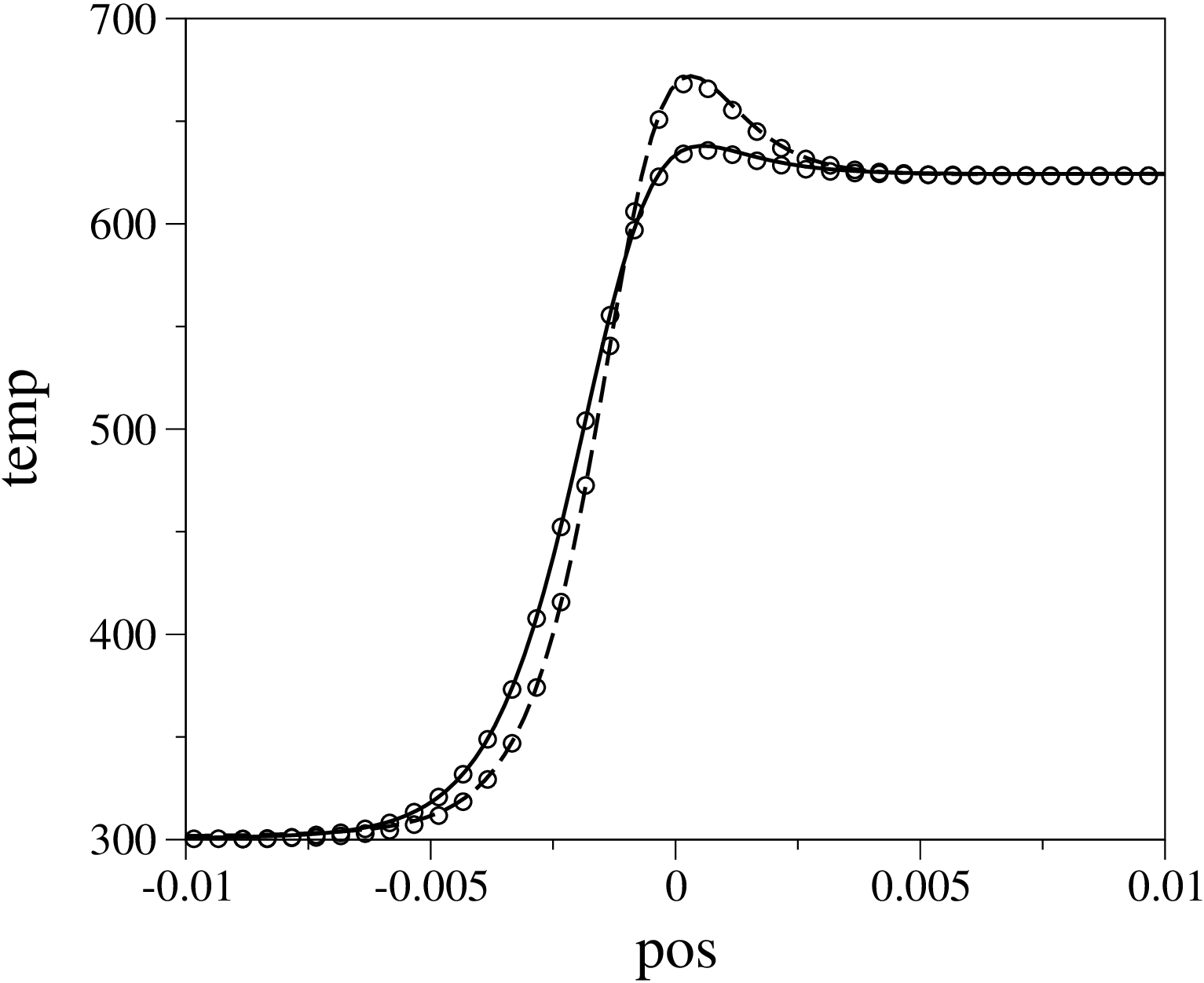}}} \,
\caption{Flow across a normal shock wave of a Ne-Ar mixture: evolution across the shock wave of the species hydrodynamic velocity and parallel temperature component (lines BESS - symbols DSMC).}\label{fig:shock_Ne_Ar_comp} 
\end{figure}   
\begin{figure}[!htbf]
\centering
\subfloat[Density.]{
\psfrag{pos}[c][c][0.8]{$\xleg$}
\psfrag{dens}[c][c][0.8]{$\rholeg$}
{\includegraphics[scale=0.285,keepaspectratio]{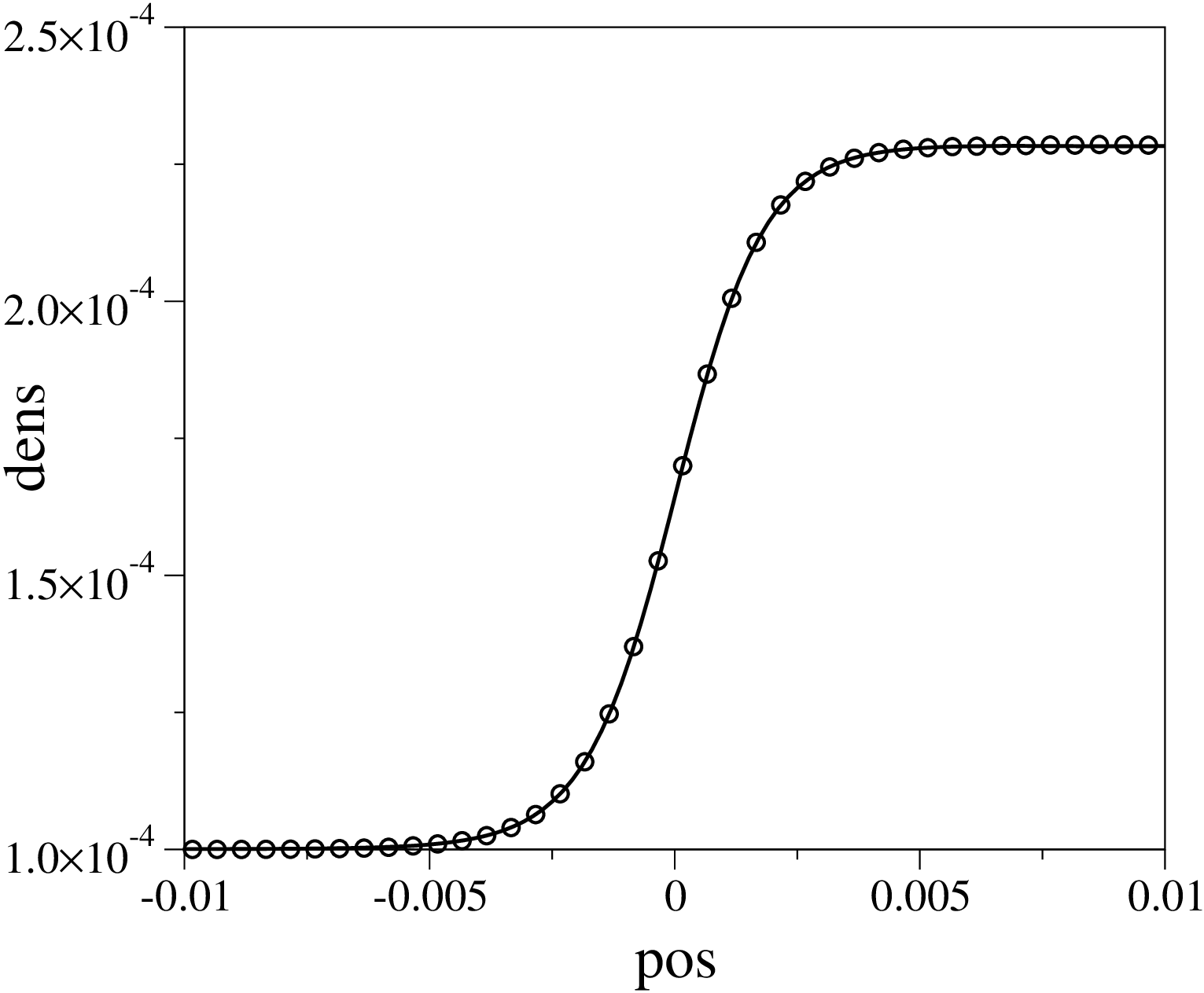}}} \,  
\subfloat[Temperature and related parallel and transverse components.]{
\psfrag{pos}[c][c][0.8]{$\xleg$}
\psfrag{temp}[c][c][0.8]{$\Tsleg$}
\psfrag{a}[l][c][0.75]{$T$}
\psfrag{b}[l][c][0.75]{$\Tx$}
\psfrag{c}[l][c][0.75]{$\Ty$}
\psfrag{d}[l][c][0.75]{$T$}
\psfrag{e}[l][c][0.75]{$\Tx$}
\psfrag{f}[l][c][0.75]{$\Ty$}
{\includegraphics[scale=0.285,keepaspectratio]{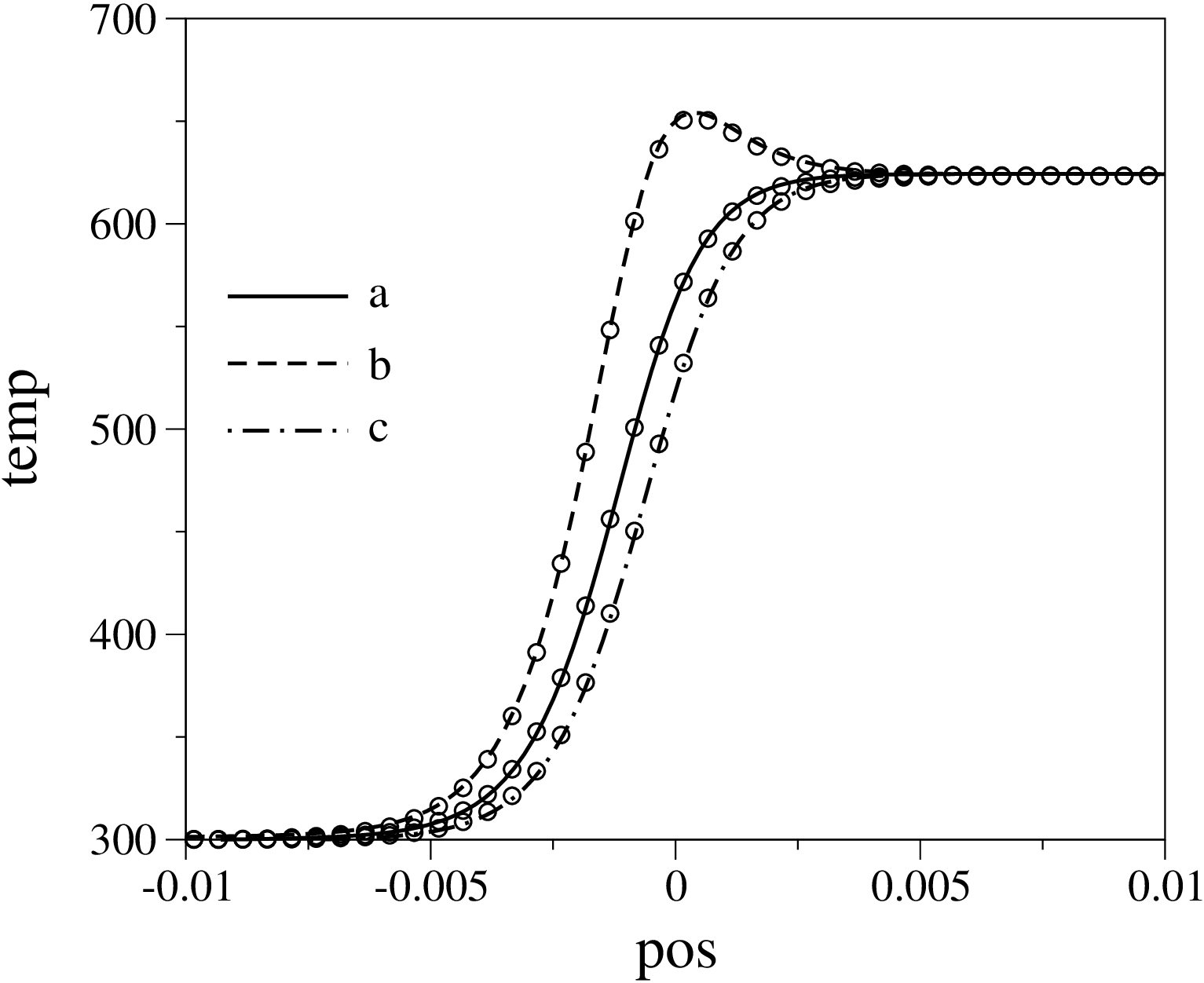}}} \,  
\caption{Flow across a normal shock wave of a Ne-Ar mixture: evolution across the shock wave of the mixture density, temperature and related parallel and transverse components (lines BESS - symbols DSMC).}\label{fig:shock_Ne_Ar_mix} 
\end{figure}
\newpage
The evolution across the shock wave of the $\vx$ axis component of the species velocity distribution function is shown in Fig. \ref{fig:shock_Ne_Ar_vdf_ev}. Due to the low value of the free-stream Mach number, small deviations from a Maxwell-Boltzmann shape are observed for the $\vx$ axis component. This justifies, in turn, the moderate maxima reached by the species parallel temperature in Fig. \ref{fig:shock_Ne_Ar_comp}. The evolution across the shock wave of the $\vy$ and $\vz$ axis components of the species distribution function (not shown in Fig. \ref{fig:shock_Ne_Ar_vdf_ev}) occurs through a sequence of Maxwell-Boltzmann distributions.   
\begin{figure}[!htbf]
\centering
\subfloat[$\Nerm$.]{
\psfrag{v}[c][c][0.8]{$\vxleg$}
\psfrag{vdf}[c][c][0.8]{$\vdfvxleg$}
\psfrag{a}[l][l][0.6]{$x = - 1 \times 10^{\, -2} \, \rm m$}
\psfrag{b}[l][l][0.6]{$x = - 1.1 \times 10^{\, -3} \, \rm m$}
\psfrag{c}[l][l][0.6]{$x = - 1 \times 10^{\, -4}\, \rm m$}
\psfrag{d}[l][l][0.6]{$x = 9 \times 10^{\, -4}\, \rm m$}
\psfrag{e}[l][l][0.6]{$x = 1.9 \times 10^{\, -3}\, \rm m$}
\psfrag{f}[l][l][0.6]{$x = 1 \times 10^{\, -2}\, \rm m$}
{\includegraphics[scale=0.285,keepaspectratio]{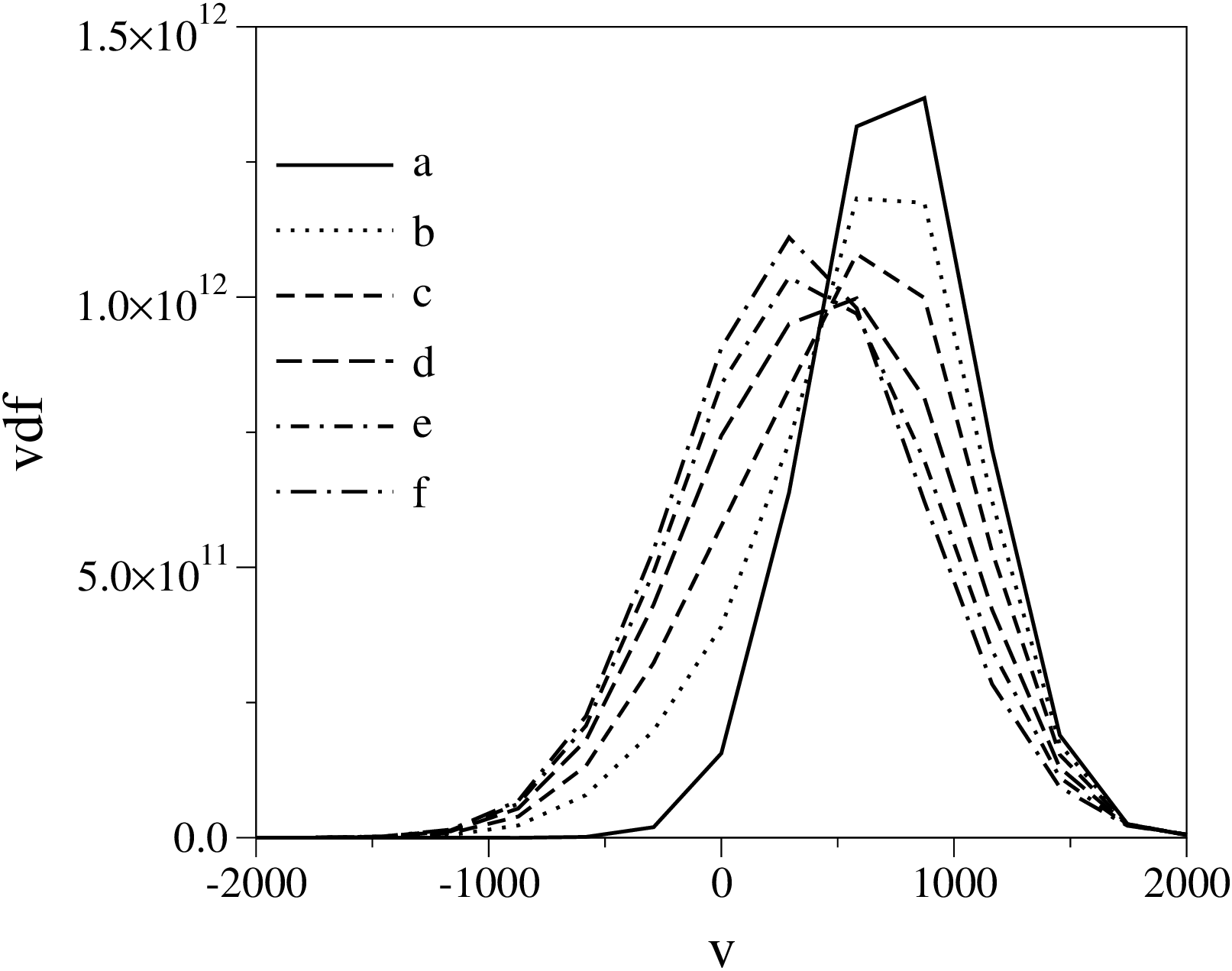}}} \,  
\subfloat[$\Arrm$.]{
\psfrag{v}[c][c][0.8]{$\vxleg$}
\psfrag{vdf}[c][c][0.8]{$\vdfvxleg$}
{\includegraphics[scale=0.285,keepaspectratio]{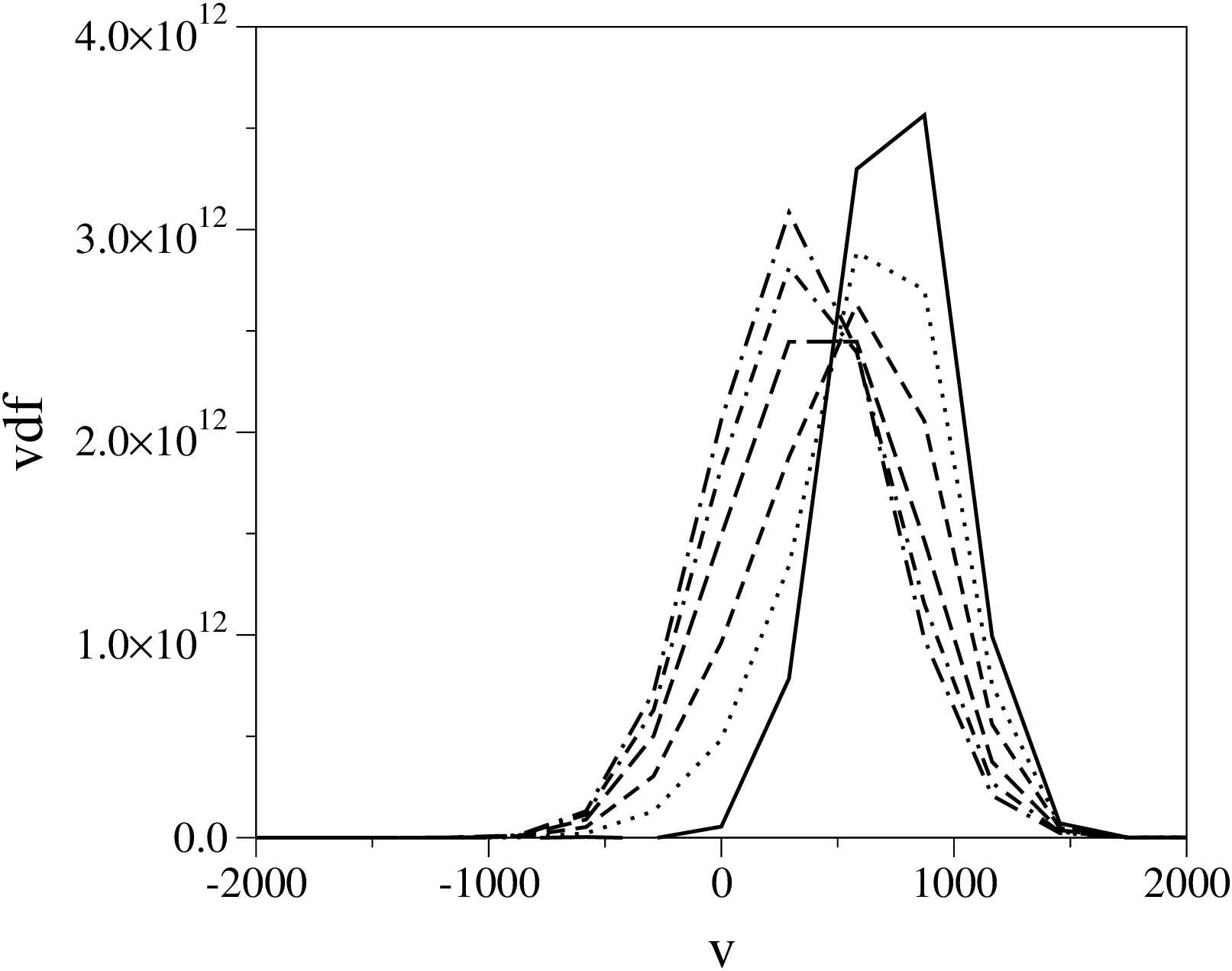}}} \,
\caption{Flow across a normal shock wave of a Ne-Ar mixture: evolution across the shock wave of the $\vx$ axis component of the species distribution function.}\label{fig:shock_Ne_Ar_vdf_ev} 
\end{figure}
\newpage
\subsection{Flow across a normal shock wave of a multi-energy level gas}
The steady-state flow across a normal shock wave of a multi-energy level gas is studied in the shock wave reference frame by considering the same physical model as that used in Sect. \ref{sec:Ar_star_Hom}. For the present calculations, only 2 energy levels are accounted for (the related values of degeneracy and energy are given in Table \ref{tab:en_lev_shock}). The total number of partial collision operators to be evaluated reduces to 16 (4 elastic and 12 inelastic).
\begin{table}[!htbf]
\begin{center}
\begin{tabular}{lll}
$i$ & $\gi$ & $\Einti \, [\mathrm{J}]$ \\
\hline
1 & 1 & 0.0 \\
2 & 1 & $4.14 \times 10^{\, -21}$ \\
\hline
\end{tabular}
\caption{Level degeneracy and energy.}
\label{tab:en_lev_shock}
\end{center}
\end{table}

The free-stream values of the gas density, temperature and velocity are $1 \times 10^{\, -4} \, \rm kg/m^{3}$, $300 \,\rm K$ and $945.33 \, \rm m/s$, respectively. The latter correspond to a free-stream Mach number equal to 3. Due to the presence of internal energy, the flow post-shock conditions are obtained by solving numerically the set of equations expressing the conservation of mass, momentum and energy fluxes between the free-stream and post-shock states (the Rankine-Hugoniot jump relations \cite{Anderson_book} cannot be applied as they are valid only for the case of a calorically perfect gas). For the present calculations, post-shock conditions are computed by using the technique suggested in \cite{Anderson_book}. The values obtained for the post-shock density, temperature and velocity for are $3.25 \times 10^{\, -4} \, \rm kg/m^{3}$, $1046.2 \, \rm K$ and $311.07 \, \rm m/s$, respectively.    
   
The numerical values of the parameters used for the discretization of the phase-space (Sect. \ref{sec:phase}) and the application of the time-marching method (Sect. \ref{sec:op_split}) are provided in Table \ref{tab:par_shock_Ar_star}.      
\begin{table}[!htbf]
\begin{center}
\begin{tabular}{llllllll}
$\Nv$ & $\Lv \, [\mrm{m/s}] $ & $\Nx$ & $\Lmx \, [\mrm{m}]$ & $\Lpx \, [\mrm{m}]$ & $\Dtc \, [\mrm{s}] $ & CFL & Limiter  \\
\hline
30    & 3400 & 201 & $2 \times 10^{\, -2}$ & $2 \times 10^{\, -2}$ & $1 \times 10^{\, -8}$ & 0.5 & van Albada \\
\hline
\end{tabular}
\caption{Simulation parameters.}
\label{tab:par_shock_Ar_star}
\end{center}
\end{table}  

\noindent As already done in Sect. \ref{sec:shock_Ne_Ar}, the position space is discretized by means of a uniform Finite volume grid. At the boundaries $x = -\Lmx$ and $x = \Lpx$, a Maxwell-Boltzmann distribution function (Eq. \eqnref{eq:MB}) corresponding, respectively, to the pre and post-shock conditions is imposed for each level. The steady-state flow across the shock wave is computed by using the same initialization procedure as in Sect. \ref{sec:shock_Ne_Ar}.
\begin{figure}[!htbf]
\centering
\subfloat[Mass fraction.]{
\psfrag{pos}[c][c][0.8]{$\xleg$}
\psfrag{mass}[c][c][0.8]{$y$}
\psfrag{a}[l][l][0.75]{$\rm Level\, 1$}
\psfrag{b}[l][l][0.75]{$\rm Level\, 2$}
{\includegraphics[scale=0.27,keepaspectratio]{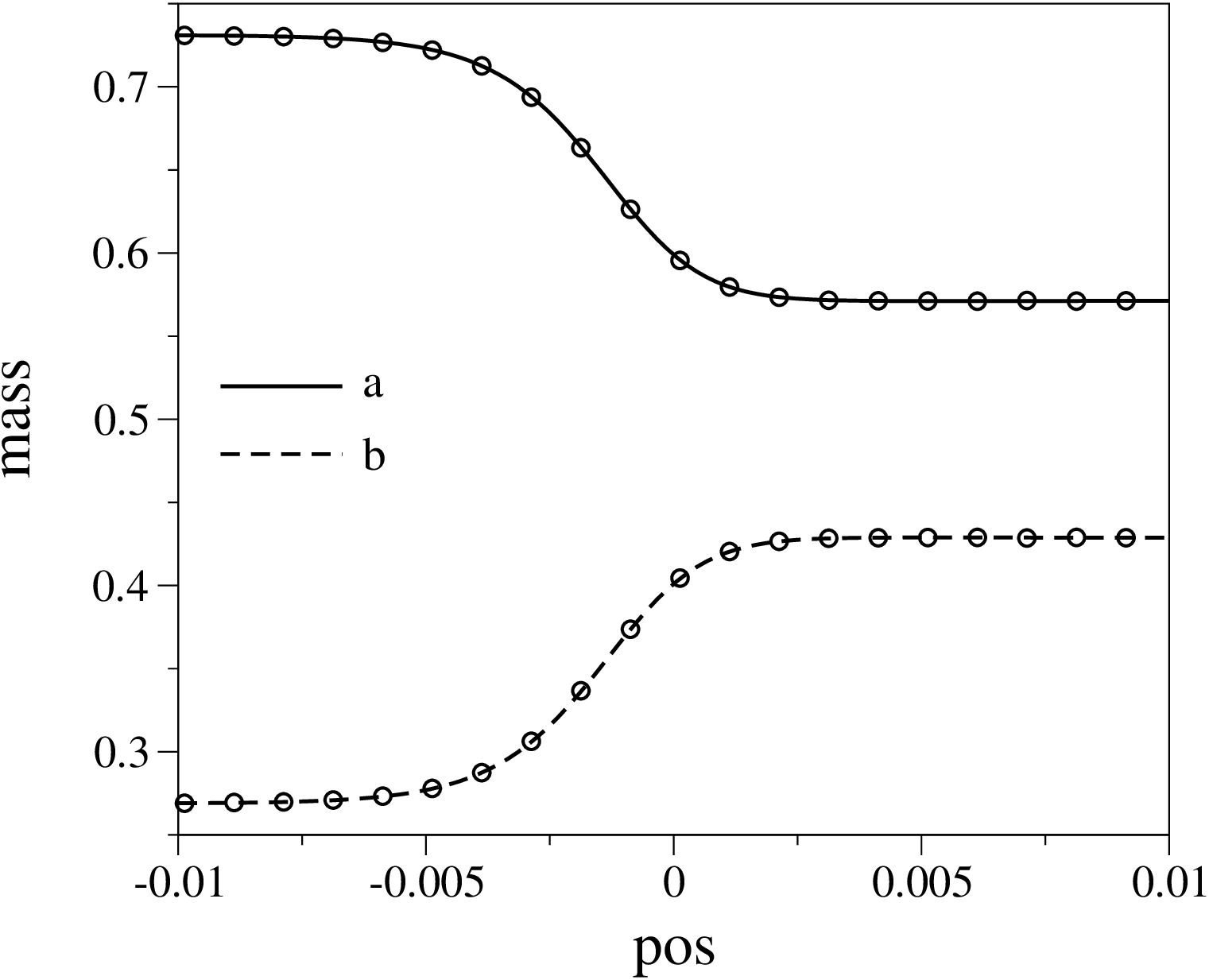}}} \,  
\subfloat[Diffusion velocity.]{
\psfrag{pos}[c][c][0.8]{$\xleg$}
\psfrag{vel}[c][c][0.8]{$\Uxleg$}
{\includegraphics[scale=0.27,keepaspectratio]{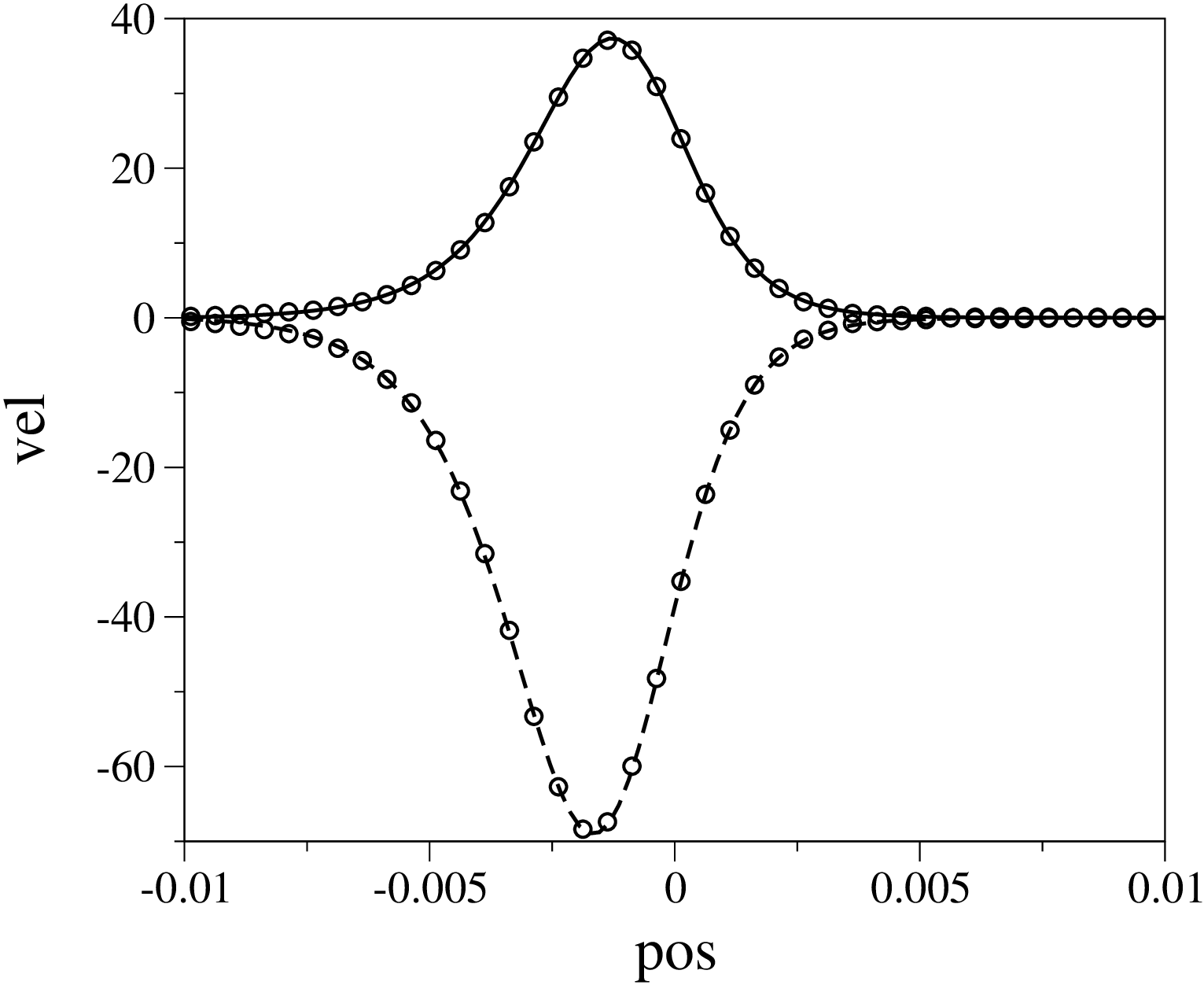}}} \,
\caption{Flow across a normal shock wave of a multi-energy level gas: evolution across the shock wave of the species mass fraction and diffusion velocity (lines BESS - symbols DSMC).}\label{fig:shock_Ar_star_species}
\end{figure}

Figure \ref{fig:shock_Ar_star_species} shows the evolution across the shock wave of the level mass fraction and diffusion velocity. The relative amount of atoms occupying a given energy level changes due to the presence of inelastic collisions. In the case when the energy levels of a chemical component are treated as separate species (like in the present case), one may say that the gas undergoes a  chemical composition variation when it crosses the shock wave. The gradients in chemical composition lead, in turn, to mass diffusion (as confirmed by the species diffusion velocity). Species separation occurs within the shock wave. However, in a comparison with the results of Sect. \ref{sec:shock_Ne_Ar}, some differences arise. In the present case, the separation is the result of chemical composition gradients caused by inelastic collisions. In the case of Sect. \ref{sec:shock_Ne_Ar}, the separation is due to the mass disparity between the species that leads, in turn, to a local chemical composition variation within the shock. The comparison with the DSMC results is again excellent. 
\begin{figure}[!htbf]
\centering
\subfloat[Pressure.]{
\psfrag{pos}[c][c][0.8]{$\xleg$}
\psfrag{pressure}[c][c][0.8]{$\pleg$}
{\includegraphics[scale=0.27,keepaspectratio]{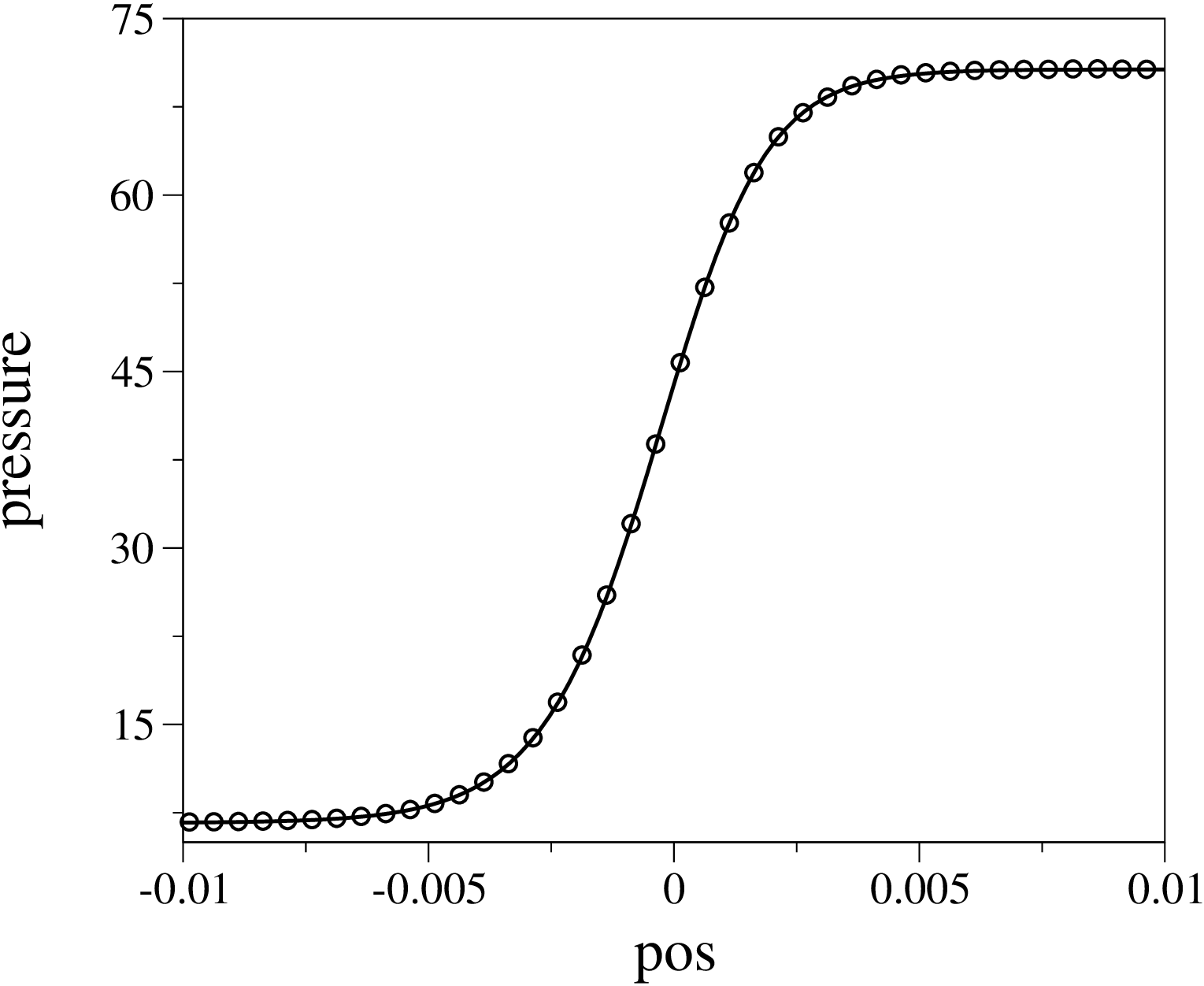}}} \,  
\subfloat[Translational and internal temperature.]{
\psfrag{pos}[c][c][0.8]{$\xleg$}
\psfrag{temp}[c][c][0.8]{$\Tslegstar$}
\psfrag{a}[l][c][0.75]{$T$}
\psfrag{b}[l][c][0.75]{$\Tx$}
\psfrag{c}[l][c][0.75]{$\Ty$}
\psfrag{d}[l][c][0.75]{$\Tint$}
{\includegraphics[scale=0.27,keepaspectratio]{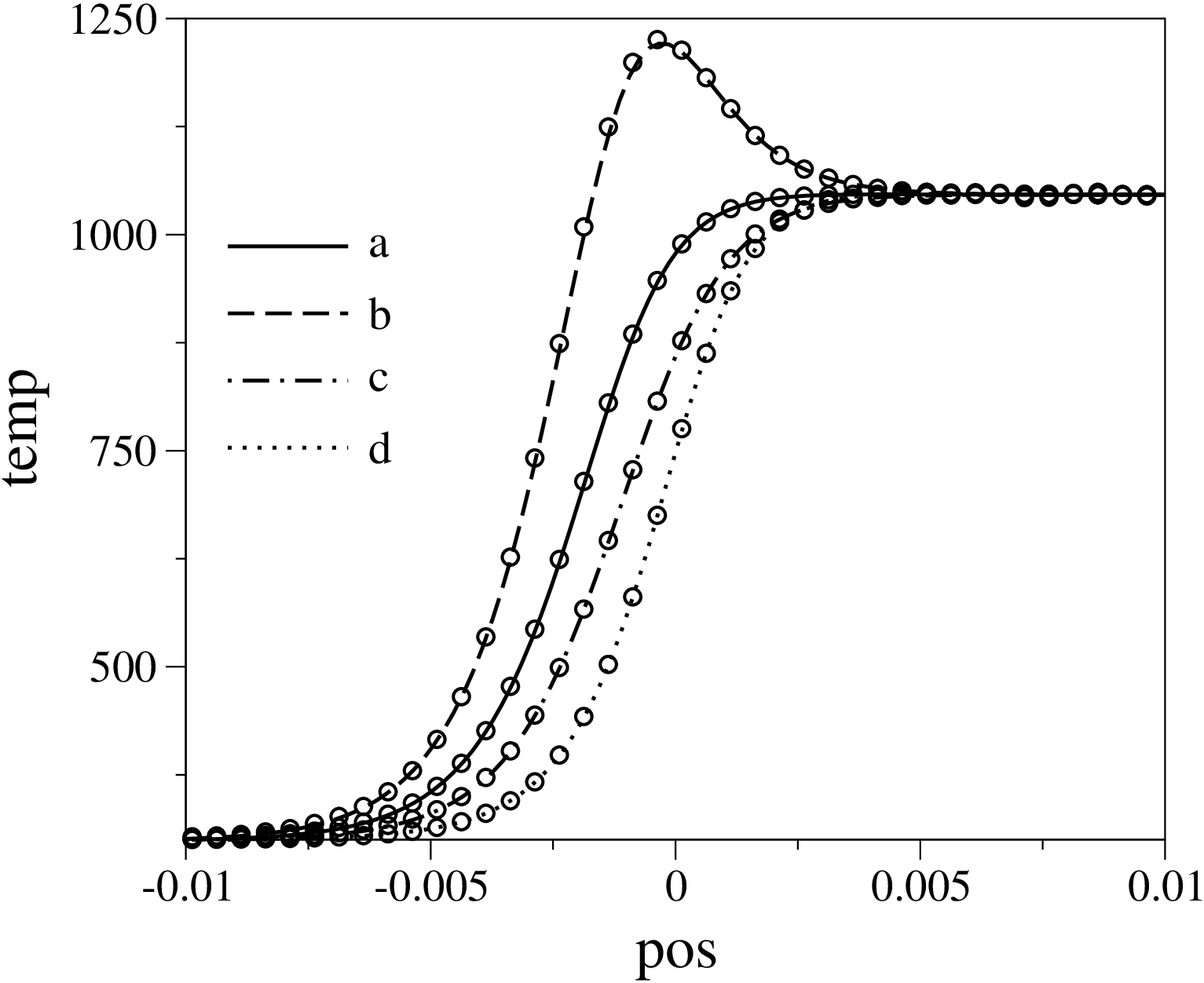}}} \,
\subfloat[Normal viscous stress.]{
\psfrag{pos}[c][c][0.8]{$\xleg$}
\psfrag{stress}[c][c][0.8]{$\tauxxleg$}
{\includegraphics[scale=0.27,keepaspectratio]{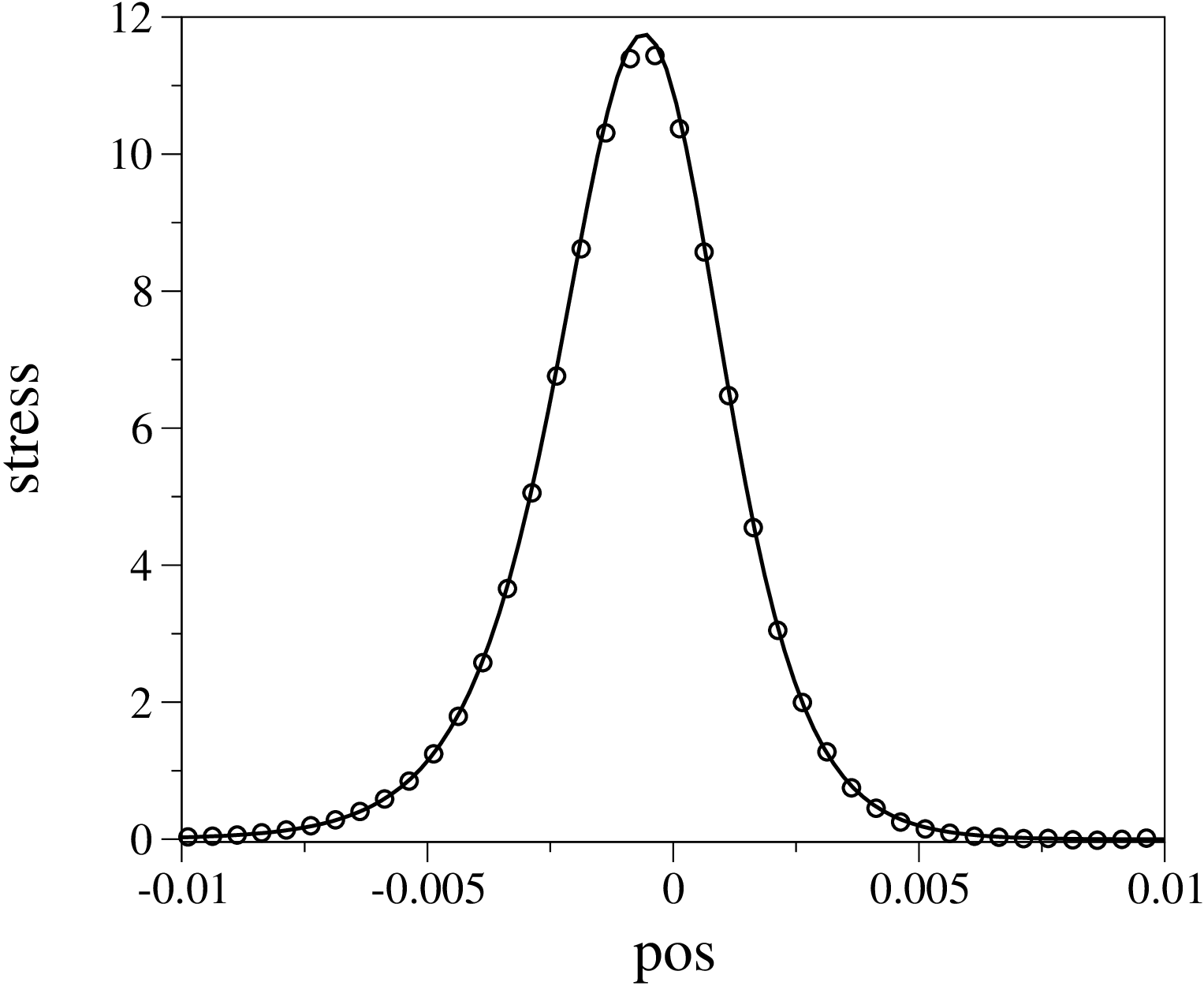}}} \,  
\subfloat[Heat flux.]{
\psfrag{pos}[c][c][0.8]{$\xleg$}
\psfrag{heat}[c][c][0.8]{$\qxleg$}
{\includegraphics[scale=0.27,keepaspectratio]{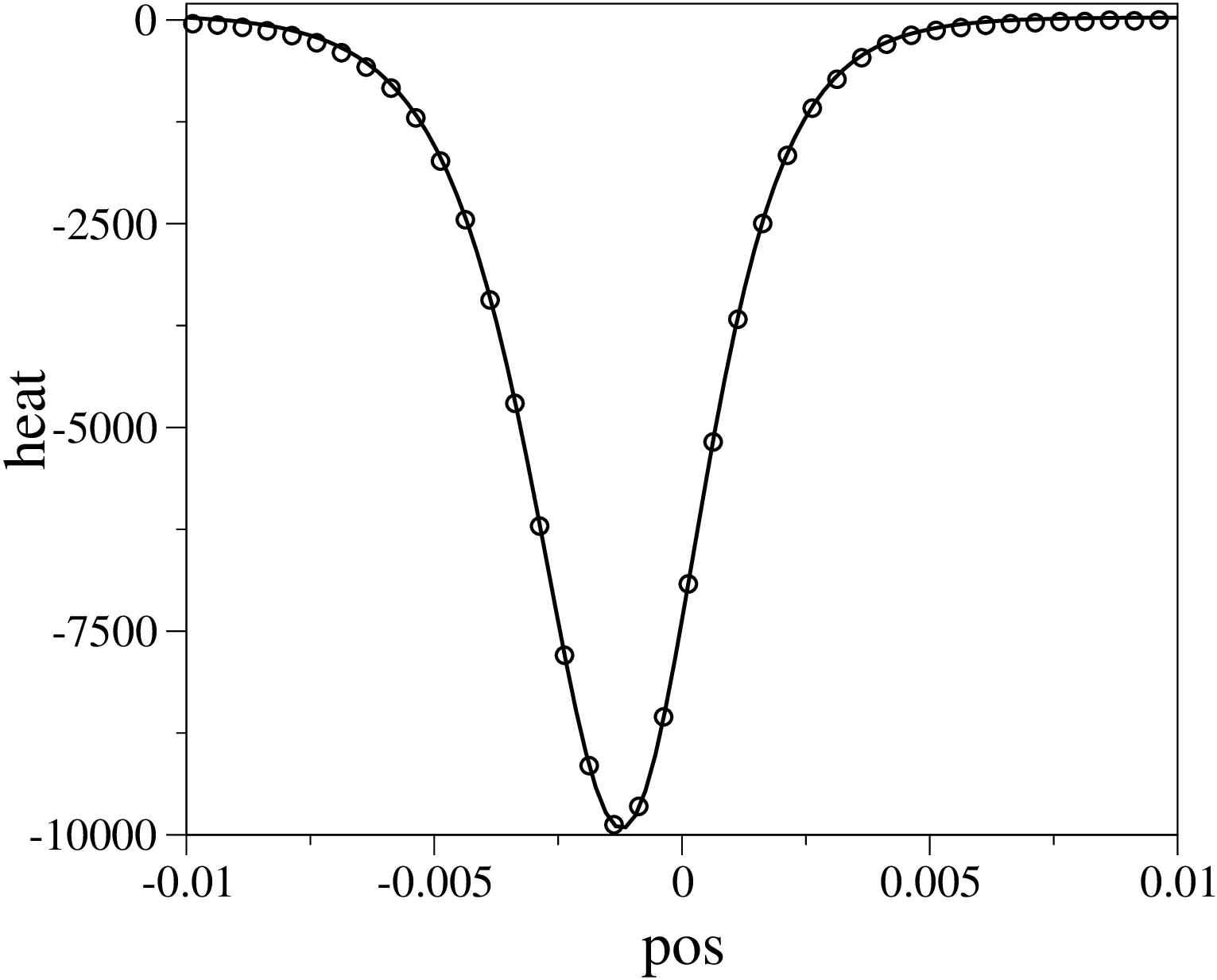}}} \,
\caption{Flow across a normal shock wave of a multi-energy level gas: evolution across the shock wave of the gas pressure, translational temperature and related parallel and transverse components, internal temperature, normal viscous stress and heat flux (lines BESS - symbols DSMC).}\label{fig:shock_Ar_star_mix} 
\end{figure}

Figure \ref{fig:shock_Ar_star_mix} shows the evolution across the shock wave of the gas pressure, translational temperature (together with the related parallel and transverse components), internal temperature, normal viscous stress and heat flux. The internal temperature lags behind the translational temperature as a result of the finite number of collisions that are needed to excite the upper internal energy levels. The parallel component of the gas translational temperature shows a pronounced maximum. As already mentioned in Sect. \ref{sec:shock_Ne_Ar}, this is due to the distortion experienced by the level velocity distribution function along the $\vx$ axis of the velocity space. The former is confirmed in Fig. \ref{fig:shock_Ar_star_vdf_ev} showing the evolution across the shock wave of the $\vx$ axis component of the level velocity distribution function. The distortions in the $\vx$ axis component are concentrated within a narrow region around the location $x = 0\, \rm m$. The evolution across the shock wave of the $\vy$ and $\vz$ axis components of the level velocity distribution function (not shown in Fig. \ref{fig:shock_Ar_star_vdf_ev}) occurs through a sequence of Maxwell-Boltzmann velocity distribution functions.       
\begin{figure}[!htbf]
\centering
\subfloat[Level 1.]{
\psfrag{v}[c][c][0.8]{$\vxleg$}
\psfrag{vdf}[c][c][0.8]{$\vdfvxleg$}
\psfrag{a}[l][l][0.6]{$x = - 2 \times 10^{\, -2} \, \rm m$}
\psfrag{b}[l][l][0.6]{$x = - 2.5 \times 10^{\, -3} \, \rm m$}
\psfrag{c}[l][l][0.6]{$x = - 1.3 \times 10^{\, -4}\, \rm m$}
\psfrag{d}[l][l][0.6]{$x = - 1 \times 10^{\, -4}\, \rm m$}
\psfrag{e}[l][l][0.6]{$x = 3 \times 10^{\, -4}\, \rm m$}
\psfrag{f}[l][l][0.6]{$x = 9 \times 10^{\, -4}\, \rm m$}
\psfrag{g}[l][l][0.6]{$x = 1.9 \times 10^{\, -3}\, \rm m$}
\psfrag{h}[l][l][0.6]{$x = 2 \times 10^{\, -2}\, \rm m$}
{\includegraphics[scale=0.285,keepaspectratio]{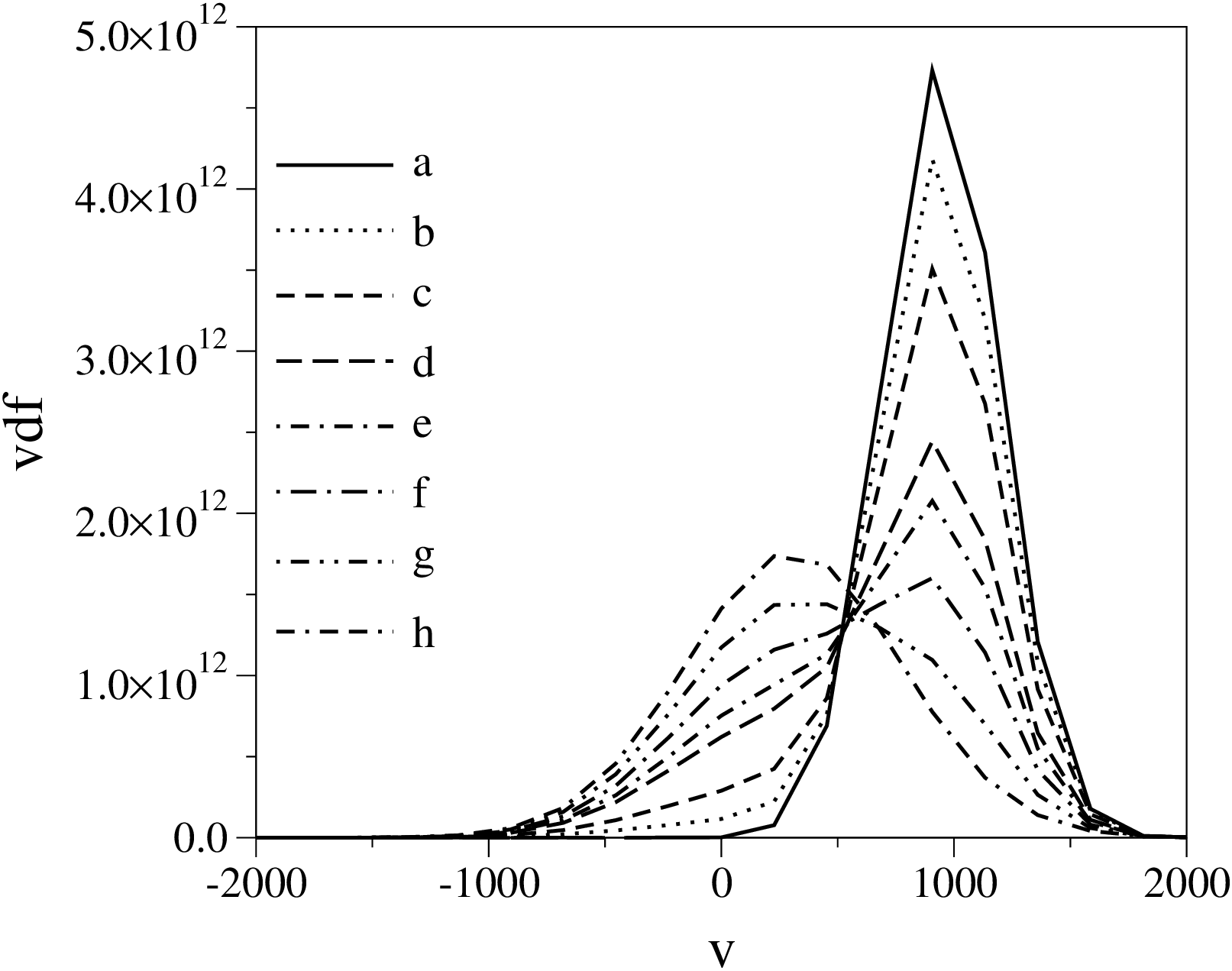}}} \,  
\subfloat[Level 2.]{
\psfrag{v}[c][c][0.8]{$\vxleg$}
\psfrag{vdf}[c][c][0.8]{$\vdfvxleg$}
{\includegraphics[scale=0.285,keepaspectratio]{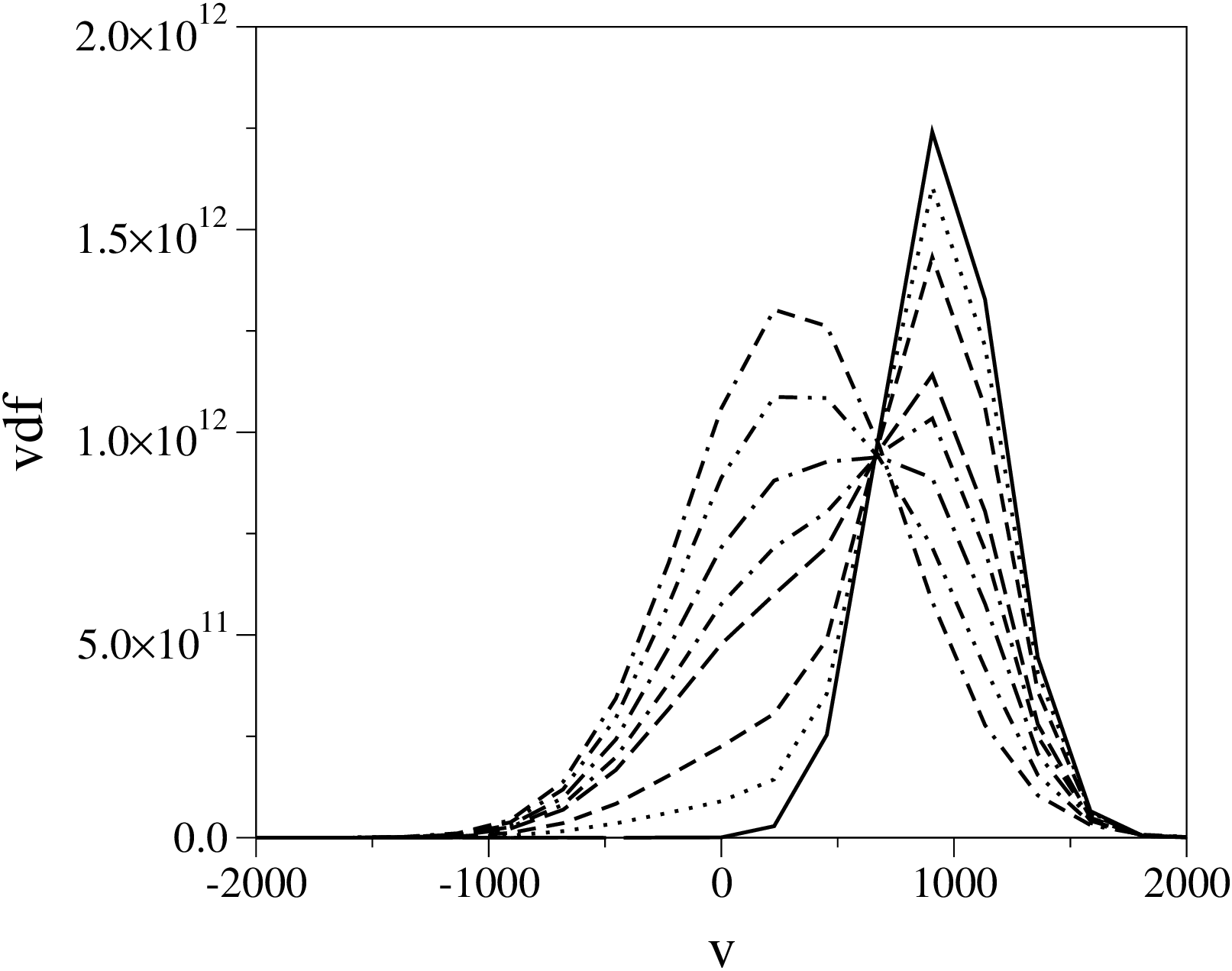}}} \,
\caption{Flow across a normal shock wave of a multi-energy level gas: evolution across the shock wave of the $\vx$ axis component of the level distribution function.}\label{fig:shock_Ar_star_vdf_ev} 
\end{figure}
\section{Conclusions}\label{sec:concl}
A spectral-Lagrangian method for the Boltzmann equation for a multi-energy level gas has been developed. The formulation of the numerical method accounts for both elastic and inelastic collisions and can also be used for the particular case of a mixture of monatomic gases without internal energy. The conservation of mass, momentum and energy during collisions is enforced through the solution of constrained optimization problems. The effectiveness of the former has been shown by the computational results obtained for both space homogeneous and space in-homogeneous problems. In all the cases, species and mixture macroscopic moments have been compared with the results obtained by means of the DSMC method. Excellent agreement has been observed. 

Future work will focus on alternative phase-space representation (such as momentum space) and on possible benefits, in terms of CPU time reduction, for cases where the velocity distribution admits certain symmetry properties in the velocity space. Computational benchmarks will be also performed by using more accurate cross-section models based on realistic interaction potentials. The results obtained will be then compared with experiments for sake of validation.      
\section*{Acknowledgements}
The authors great-fully acknowledge Mr. Erik Torres at von Karman Institute for the useful discussions on the simulations presented in this paper and for providing the DSMC results used for verification. Research of Alessandro Munaf\`{o} and Thierry E. Magin is sponsored by the European Research Council Starting Grant $\# 259354$, research of Jeffrey R. Haack is sponsored by the NSF Grant $\#\mrm{DMS} - 0636586$ and research of Irene M. Gamba is sponsored by the NSF Grants $\#\mrm{DMS} - 1109625$ and $\#\mrm{DMS} - 1107465$. 

\appendix
\section{Numerical evaluation of the Fourier and inverse Fourier transform}\label{sec:FFT}
Let $f=f(\Vi)$ be a function of the velocity $\Vi$ and let $\ghat = \ghat (\fa)$ be a function of the Fourier variable $\fa$. According to the definitions introduced in Sect. \ref{sec:phys_Fou}, the Fourier transform of the function $f$ and the inverse Fourier transform of the function $\ghat$ are:   
\begin{IEEEeqnarray}{rCl}
\fhat (\fa) & = & \scale \!\! \DomI \!\!\! \exp\left(- \imath \, \fa \cdot \Vi\right) \, f (\Vi) \, d\Vi, \quad \fa \, \in \, \Rd, \label{eq:FT}\\
g  (\Vi)      & = & \scale \!\! \DomIiF \!\!\! \exp\left(\imath \, \fa \cdot \Vi \right) \, \ghat(\fa)  \, d\fa, \quad \Vi \, \in \, \Rd. \label{eq:iFT}
\end{IEEEeqnarray}
The integrals in Eqs. \eqnref{eq:FT} - \eqnref{eq:iFT} must be replaced with discrete sums because of the discretization of the velocity space introduced in Sect. \ref{sec:phase}.

The substitution of the Eq. \eqnref{eq:vk} and Eq. \eqnref{eq:epsk} for $\vk$ and $\ze$, respectively, in Eqs. \eqnref{eq:FT} - \eqnref{eq:iFT} and the replacement of continuous integrals with discrete sums, leads to:
\begin{IEEEeqnarray}{rCl}
\fhat (\ze) & = & \scale \!\! \sumkvec \!\! \wk \exp\left(- \imath \, \ze \cdot \vk \right) \, f (\vk) \, \Dvc, \quad \ze \, \in \, \VcalF, \label{eq:FTd}\\
g (\vk) & = & \scale \!\! \sumepsvec \!\! \weps \exp\left(\imath \, \ze \cdot \vk \right) \, \ghat (\ze) \, \Detac, \quad \vk \, \in \, \Vcal,  \label{eq:iFTd}
\end{IEEEeqnarray}
where the global integration weights $\wk$ and $\weps$ associated to the discrete velocity node $\vk$ and the discrete Fourier velocity node $\ze$, respectively, are $\wk = \wkx \wky \wkz$ and $\weps = \wepsx \wepsy \wepsz$. The expansion of the dot product $ \ze \cdot \vk$ in Eqs. \eqnref{eq:FTd}-\eqnref{eq:iFTd} gives:
\be \label{eq:dot} 
\ze \cdot \vk = (- \Lv + \kx \, \Dv) \, (- \Leta + \epsx \, \Deta) + (- \Lv + \ky \, \Dv) \, (- \Leta + \epsy \, \Deta) + (- \Lv + \kz \, \Dv) \, (- \Leta + \epsz \, \Deta). 
\ee 
After some algebraic manipulation and the use of the relation $\Dv \, \Deta = 2 \, \pi /\Nv$ (Eq. \eqnref{eq:FFT}), Eq. \eqnref{eq:dot} can be re-written as:
\be \label{eq:dot2}
\ze \cdot \vk = 3 \, \Lv \, \Leta - \Lv \, \Deta \, (\epsx + \epsy + \epsz) - \Leta \, \Dv \, (\kx + \ky + \kz) + \fr{2 \, \pi}{\Nv} (\kvec \cdot \epsvec).
\ee
The substitution of Eq. \eqnref{eq:dot2} in Eqs. \eqnref{eq:FTd} - \eqnref{eq:iFTd} gives: 
\begin{IEEEeqnarray}{rCl}
\fhat (\ze) & = & \fr{\exp\left[- \imath \, \delta (\epsvec)  \right]}{\pifac} \sumkvec \!\! \fstar (\vk) \exp\left[- \imath \, \fr{2 \, \pi}{\Nv} (\kvec \cdot \epsvec ) \right], \quad \ze \, \in \, \VcalF, \label{eq:FTd2}\\
g (\vk) & = & \fr{\exp\left[ \imath \, \gamma (\kvec) \right]}{\pifac} \sumepsvec \!\! \ggstar (\ze) \exp\left[\imath \, \fr{2 \, \pi}{\Nv} (\kvec \cdot \epsvec) \right], \quad \vk \, \in \, \Vcal. \label{eq:iFTd2}
\end{IEEEeqnarray}
The quantities $\delta (\epsvec)$ and $\gamma (\kvec)$ in the exponential in front of the sums in Eqs. \eqnref{eq:FTd2} - \eqnref{eq:iFTd2} are:
\begin{IEEEeqnarray}{rCl}
\delta (\epsvec) & = & \Lv \left[3 \, \Leta - \Deta  \, (\epsx + \epsy + \epsy)\right], \quad \epsvec \, \in \, \Vsetc,\\
\gamma (\kvec) & = & \Leta \left[3 \, \Lv - \Dv \, (\kx + \ky + \kz) \right], \quad \, \kvec \in \, \Vsetc,
\end{IEEEeqnarray}
while the functions $\fstar (\vk)$ and $\ggstar (\ze)$ in the same equations are defined as:
\begin{IEEEeqnarray}{rCl}
\fstar (\vk) & = & \wk \,  f (\vk) \exp\left[\imath \, \Leta \, \Dv \, (\kx + \ky + \kz) \right] \Dvc, \quad \vk \, \in \, \Vcal, \label{eq:g*} \\
\ggstar (\ze) & = & \weps \,  \ghat (\ze) \exp\left[- \imath \, \Lv \, \Deta \, (\epsx + \epsy + \epsz)\right] \Detac, \quad \ze \, \in \, \VcalF. \label{eq:h*}
\end{IEEEeqnarray}
The sums in Eqs. \eqnref{eq:FTd2} - \eqnref{eq:iFTd2} correspond, respectively, to the definitions of the Fast-Fourier-Transform ($\FFT$) and inverse Fast-Fourier-Transform ($\iFFT$ of inverse $\FFT$) of the functions $\fstar$ and $\ggstar$ (with no scaling):
\begin{IEEEeqnarray}{rCl}
\FFT (\fstar)(\ze) & = & \!\! \sumkvec \!\! \fstar (\vk) \exp\left[- \imath \, \fr{2 \, \pi}{\Nv} (\kvec \cdot \epsvec) \right], \quad \ze \, \in \, \VcalF,  \label{eq:FFTg}\\
\iFFT (\ggstar)(\vk) & = & \!\! \sumepsvec \!\! \ggstar (\ze) \exp\left[\imath \, \fr{2 \, \pi}{\Nv} (\kvec \cdot \epsvec) \right], \quad \vk \, \in \, \Vcal. \label{eq:iFFTh}
\end{IEEEeqnarray}
In order to exploit Eqs. \eqnref{eq:FFTg} - \eqnref{eq:iFFTh} for computing the Fourier and the inverse Fourier transform, the following algorithm is proposed:
\begin{enumerate}
  \item Given the discrete values of the function $g$ (or $\ghat$), the function $\fstar$ (or $\ggstar$) is evaluated by means of Eq. \eqnref{eq:g*} (or Eq. \eqnref{eq:h*}).
  \item The FFT of $\fstar$ (or the inverse FFT of $\ggstar$) is computed by means of Eq. \eqnref{eq:FFTg} (or Eq. \eqnref{eq:iFFTh}).
  \item The result obtained is substituted in Eq. \eqnref{eq:FTd2}) (or Eq. \eqnref{eq:iFTd2}). 
\end{enumerate}
In the present work, the computation of the FFT and the inverse FFT of functions is performed by means of the FFTW (Fastest-Fourier-Transform in the West) package \cite{fftw}.
\section{Numerical evaluation of the weighted convolution}\label{sec:wconv}
The continuous integrals defining the weighted convolutions in Eqs. \eqnref{eq:FT_theor} - \eqnref{eq:FT_theor2} are approximated as follows. Let $\boldsymbol{\kappa} = (\kappa_x, \kappa_y, \kappa_z)$ and $\wkappaset= (\wkappax, \wkappay, \wkappaz)$ be, respectively, the vector of indices and the vector of integration weights associated to the discrete Fourier velocity node $\xik$. The Fourier transform of the partial elastic and inelastic collision operators ($\Qel$ and $\Qin$, respectively) evaluated at the discrete Fourier velocity node $\ze$ become:
\begin{IEEEeqnarray}{rCl}
\FQel (\ze) & = & \scale \sum_{\boldsymbol{\kappa} \, \in \, \mathcal{I}^{\, *}_{\, \kappavec}} \!\! \wkappa \Fvdfi \left(\ze - \xik \right) \, \Fvdfj (\xik) \, \FFGel \left(\ze,\xik\right) \, \Detac, \quad \ijset, \label{eq:disc_convel} \\
\FQin (\ze) & = & \scale \sum_{\boldsymbol{\kappa}\, \in \, \mathcal{I}^{\, *}_{\, \kappavec}} \!\! \wkappa  \left[ \Fvdfh \left(\ze - \xik \right) \, \Fvdfl (\xik) \, \FFGinG \left(\ze,\xik\right) - \Fvdfi \left(\ze - \xik \right) \, \Fvdfj (\xik) \, \FFGinL \left(\xik\right) \right] \Detac, \nonumber \\
&&\ijhlsetC , \, \ze \, \in \, \VcalF. \label{eq:disc_convin}
\end{IEEEeqnarray} 
In Eqs. \eqnref{eq:disc_convel} -  \eqnref{eq:disc_convin}, the quantity $\wkappa = \wkappax \wkappay \wkappaz$ is the global integration weight associated to the discrete Fourier velocity node $\xik$, while the set $\mathcal{I}^{\, *}_{\, \kappavec}$ is defined as:
\be \label{eq:lim_ind}
\mathcal{I}^{\, *}_{\, \kappavec} = \left\{(\kappax^{\, -},\kappax^{\, +}) \times (\kappay^{\, -},\kappay^{\, +}) \times (\kappaz^{\, -},\kappaz^{\, +})\right\}\, \subset \, \Vsetc. 
\ee
In Eq. \eqnref{eq:lim_ind}, the $-$ and the $+$ upper-scripts are used to indicate, respectively, the lower and the upper limits for the indices $\kappax$, $\kappay$ and $\kappaz$ associated to the discrete Fourier velocity node $\xik$ and are computed based on the following relations:
\begin{IEEEeqnarray}{rCl} 
\kappa_{\, \mrm{s}}^{\, -} & = & \begin{cases}
                 0 & \mathrm{if \,\,} \varepsilon_s < \Nv/2, \\
                 \varepsilon_{\, \alpha} - \Nv/2 + 1 & \mathrm{if \,\,} \varepsilon_{\, \alpha} \geq \Nv/2,
                 \end{cases} \\
\kappa_{\, \mrm{s}}^{\, +}  & = & \begin{cases}
                 \varepsilon_{\, \alpha} + \Nv/2 - 1 & \mathrm{if \,\,} \varepsilon_{\, \alpha} < \Nv/2, \\
                 \Nv & \mathrm{if \,\,} \varepsilon_{\, \alpha} \geq \Nv/2, \quad \alset.
                 \end{cases} 
\end{IEEEeqnarray} 
The introduction of the above lower and upper limits on the $\kappax$, $\kappay$ and $\kappaz$ indices is equivalent to set to zero the functions $\Fvdfh$ and $\Fvdfi$, respectively, in the discrete sums given in Eqs. \eqnref{eq:disc_convel} - \eqnref{eq:disc_convin} when their argument $\left(\ze - \xik \right)$ goes beyond the limits of the Fourier velocity space (Eq. \eqnref{eq:epsk}).
\section{Macroscopic moments}\label{app:mom}  
The macroscopic moments defined in Sects. \ref{sec:mom_cons} and \ref{sec:neq} are approximated as follows:  
\bi 
   \item Species density: 
    \be 
       \rhoi = \mi \!\! \sumkvec \!\! \weight \, \vdfi(\vk) \, \Dvc, \quad \iset.
    \ee 
   \item Hydrodynamic velocity:
    \be 
       \Vvec = \fr{1}{\rho} \sumins \sumkvec \!\! \weight \,  \mi \, \vk \, \vdfi(\vk)  \, \Dvc.
    \ee
   \item Species diffusion velocity:
    \be 
       \Vdiffi = \fr{1}{\nbi} \sumkvec \!\! \weight \, (\vk  - \Vvec) \vdfi(\vk)  \, \Dvc, \quad \iset.
    \ee
   \item Species translational temperature components: 
    \be
      \Tial = \fr{\mi}{\nbi \, \kb} \! \sumkvec \!\! \weight \left[v_{h_{\alpha}} - \Val \right]^{\, 2} \vdfi(\vk) \, \Dvc, \quad \iset, \, \alset.  
    \ee
    \item Viscous stress tensor: 
     \be 
        \taumix = \!\! \sumins \sumkvec \!\! \weight \, \mi \left(\vk - \Vvec \right) \otimes \left(\vk - \Vvec \right) \vdfi(\vk) \, \Dvc - p \, \mbf{I}.
     \ee
    \item Heat flux vector:
    \be
       \qmix = \!\! \sumins \sumkvec \!\! \weight \, \left(\vk - \Vmix \right) \left( \half \mi \left| \vk - \Vvec \right|^{\, 2} + \Einti\right) \vdfi(\vk) \, \Dvc.
     \ee
\ei
\bibliographystyle{plain}
\bibliography{ref}
\end{document}